\documentclass[aop,preprint]{imsart}

\RequirePackage[OT1]{fontenc}
\RequirePackage{amsthm,amsmath}
\RequirePackage[numbers]{natbib}
\RequirePackage[colorlinks,citecolor=blue,urlcolor=blue]{hyperref}

\usepackage[all]{xy}
\usepackage{dsfont}
\usepackage{amssymb}

\newcommand\NN{\mathbb{N}}

\newcommand\RR{\mathbb{R}} 
\newcommand\CC{\mathbb{C}} 
\newcommand\PP{\mathbb{P}}
\newcommand\EE{\mathbb{E}}

\newcommand\tr{\mathrm{tr}}
\newcommand\Car{\mathds{1}}
\newcommand\eps{\varepsilon}

\newcommand\Var{\mathrm{Var}}

\newcommand\rk{\mathrm{rank}}

\newcommand\supp{\mathrm{supp}}

\startlocaldefs
\numberwithin{equation}{section}

\theoremstyle{definition} 
\newtheorem{Def}{Definition}[section] 
\theoremstyle{plain} 
\newtheorem{Pro}[Def]{Proposition} 
\newtheorem{Lem}[Def]{Lemma}  
\newtheorem{lemma}[Def]{Lemma} 
\newtheorem{The}[Def]{Theorem} 
\newtheorem{fact}[Def]{Fact} 
\newtheorem{Cor}[Def]{Corollary}  
 
\theoremstyle{remark}

\newtheorem{Rem}[Def]{Remark}

\endlocaldefs

\begin{document}

\begin{frontmatter}
\title{Nonlinear large deviation bounds with applications to Wigner matrices and sparse Erd\H{o}s--R\'enyi graphs}
\runtitle{Nonlinear large deviation bounds}

\begin{aug}
\author{\fnms{Fanny} \snm{Augeri}\thanksref{t1}\ead[label=e1]{fanny.augeri@weizmann.ac.il}},

\thankstext{t1}{This project has received funding from the European Research Council (ERC) under the European Union’s Horizon 2020 research and innovation programme (grant agreement No. 692452).}
\runauthor{F. Augeri}

\affiliation{Weizmann Institute of Science}

\address{
Faculty of Mathematics and Computer Science\\
 The Weizmann Institute of Science\\
234 Herzl Street\\
 Rehovot 7610001 Israel\\
\printead{e1}\\
\phantom{E-mail:\ }}
\end{aug}

\begin{abstract}
We prove general nonlinear large deviation estimates similar to Chatterjee-Dembo's original bounds except that we do not require any second order smoothness. Our approach relies on convex analysis arguments and is valid for a broad class of distributions. Our results are then applied in three different setups. Our first application consists in the mean-field approximation of the partition function of the Ising model under an optimal assumption on the spectra of the adjacency matrices of the sequence of graphs. Next, we apply our general large deviation bound to investigate the large deviation of the traces of powers of Wigner matrices with sub-Gaussian entries, and the upper tail of cycles counts in sparse Erd\H{o}s--R\'enyi graphs down to the sparsity threshold $n^{-1/2}$.
\end{abstract}



\end{frontmatter}

\section{Introduction}
We will be concerned with the following large deviation question: given a random vector $X$ in $\RR^n$ and a smooth function $f : \RR^n \to \RR$, when is \textit{shifting the mean of $X$} the \textit{optimal} large deviation strategy for the random variable $f(X)$?  

In a seminal paper \cite{CD}, Chatterjee and Dembo provided a sufficient criterion, and showed in the case where $X$ is uniformly sampled on the discrete hypercube, that as soon as $f$ has a gradient of \textit{low complexity}, in the sense that it can be encoded by a small number of bits compared to the dimension, then the large deviations of $f(X)$ are only due to changes of the mean of $X$. Their framework encompasses a large class of large deviation problems as the mean-field approximation of the partition function of the Ising model (see \cite{BM}), the large deviation of sub-graph counts in  Erd\H{o}s--R\'enyi graphs - which were until then tackled using the graphon formalism in \cite{CV} - and arithmetic progressions. Later, Yan generalized in \cite{Yan} their nonlinear large deviation estimate to any compactly supported distribution. 

As a consequence of powerful structure theorems for probability measures on the discrete hypercube, Eldan obtained in \cite{Eldan} nonlinear large deviation bounds which do not require any second order smoothness on the given function, and where the complexity of the gradient is assessed in terms of the Gaussian mean-width of its image. Using these, he partially recovered a known result of Basak and Mukherjee \cite{BM} on the mean-field approximation of the partition function of the Ising model, and improved the threshold of sparsity for the large deviation upper tail  of triangle counts.

We will improve the error terms in the nonasymptotic bounds of \cite{CD}, and in particular remove the smoothness terms.  The motivation behind these improvements is that it allows for obtaining weaker dimension dependence. In particular, this entails that one can consider large deviation speeds much smaller than the dimension. In turn, this allows us to reach the critical sparsity level - as identified in \cite{LZ} and \cite{BGLZ} - in the large deviation bounds of cycle counts in sparse  Erd\H{o}s--R\'enyi graphs.

We will develop nonlinear large deviations estimates for general distributions, and propose several applications.
As a first example, we will apply our results to the mean-field approximation of the partition function of the Ising model on a sequence of graphs whose spectra satisfy certain assumptions which include star graphs, thus strengthening the previous result of Basak and Mukherjee \cite{BM}.

Next, we will use our nonlinear large deviation estimate to investigate the large deviation of traces of Wigner matrices with sub-Gaussian entries. Using a truncation argument to reduce the complexity of our function at hand, we show general upper and lower bounds. We will also prove the universality of the rate function for a class of Wigner matrices with \textit{sharp sub-Gaussian tails}. This complements the results of Guionnet and Husson  \cite{GH}, who introduced this class and showed such universality for the large deviation of the largest eigenvalue.

Finally, we estimate the upper tail of the large deviations of cycle counts in sparse  Erd\H{o}s--R\'enyi graphs, down to the sparsity threshold $n^{-1/2}$ - up to logarithmic corrections -, improving in the case of triangles the recent result of Cook and Dembo \cite{CoD}.

\subsection{Main results}
We begin with introducing some definitions.
    Let $\mu$ be a probability measure on $\RR^n$ whose support is not included in a hyperplane. 
Let $\Lambda_\mu$ be the \textit{logarithmic Laplace transform} of $\mu$, that is,
$$\forall \lambda \in \RR^n, \  \Lambda_\mu (\lambda) = \log \int e^{\langle \lambda, x\rangle} d\mu(x).$$
It is known that $\Lambda_\mu$ is a strictly convex function which is $C^{\infty}$ on the interior of its domain denoted by $\mathcal{D}_{\Lambda_\mu}$.
We define  $\Lambda^*_\mu$, the \textit{Legendre transform} of $\Lambda_\mu$, by
$$\forall x \in \RR^n, \ \Lambda^*_\mu(x) = \sup_{\lambda \in \RR^n} \{ \langle \lambda, x\rangle - \Lambda_\mu(\lambda)\}.$$
Following \cite[Section 26]{Rockafeller}, we say that a convex function $\Lambda : \RR^n \to \RR\cup\{+\infty\}$ is \textit{essentially smooth} if the interior of its domain  $\mathrm{int}(\mathcal{D}_{\Lambda})$ is non-empty, $\Lambda$ is differentiable on $\mathrm{int}(\mathcal{D}_{\Lambda})$, and steep, that is, for any $\lambda_k \in \mathrm{int}(\mathcal{D}_{\Lambda})$, converging to a point on the boundary of $\mathcal{D}_{\Lambda}$,
$$   \lim_{k\to +\infty} || \nabla \Lambda(\lambda_k)|| = +\infty.$$
The interest of this concept of essential smoothness lies in the fact that the Legendre transformation leaves invariant the class of strictly convex functions on $\RR^n$  which are essentially smooth. 
More precisely, assuming that $\Lambda_\mu$ is \textit{essentially smooth}, then by \cite[Theorem 26.5]{Rockafeller}, we know that $\Lambda^*_\mu$ is also an essentially smooth function, the map $\nabla \Lambda_\mu$ is one-to-one onto $\mathrm{int}(\mathcal{D}_{\Lambda^*_\mu})$, the interior of the domain of $\Lambda^*_\mu$, and the maps $\nabla \Lambda_\mu$ and $\nabla \Lambda^*_{\mu}$ are inverse maps from one another.

For any $\lambda \in \mathrm{int}(\mathcal{D}_{\Lambda_\mu})$, one can define the measure,
\begin{equation} \label{defmuh} \mu_y =e^{\langle \lambda ,x\rangle -\Lambda_\mu(\lambda)} d\mu(x),\end{equation}
where $y =\nabla \Lambda_\mu(\lambda)$. Observe that differentiating $\Lambda_\mu$, one obtains the barycenter of $\mu_y$, that is,
\begin{equation}\label{barycentermuh} y =  \nabla \Lambda_\mu(\lambda) =\int x d\mu_y(x).\end{equation}
Thus, the family of measures $\mu_y$ represents the collection of tilts of $\mu$, indexed by their barycenters.

Our first result is the following non-asymptotic version of Varadhan's lemma (see \cite[Theorem 4.3.1]{DZ}). This will be instrumental in our treatment of the mean-field behavior of the Ising model, see \S \ref{Ising}.

\begin{The}\label{partitionfunc}
Assume $\mu$ is compactly supported. Denote by $K$ the convex hull of its support, and by $D$ its diameter.
Let $f : \RR^n \to \RR$ be a continuously differentiable function.
 Let $\delta>0$ and $\mathcal{D}_{\delta}$ be a $\delta/D$-net of the convex hull of $\nabla f(K)$ for the $\ell^2$-norm.  Then,
$$\log \int e^{f} d\mu \leq \sup_{y \in \RR^n}  \{ f(y) -\Lambda^*_\mu(y)\}  + \log |\mathcal{D}_{\delta}| +  \delta.$$
\end{The}

Our approximation of the partition function is similar to the one of Chatterjee-Dembo (see \cite[Theorem 1.6]{CD}), except that we do not have any other error terms but the one coming from the cardinality of the net $\mathcal{D}_\delta$. 	Unlike Eldan's result (see \cite[Corollary 2]{Eldan}) which involves the \textit{Gaussian mean-width} of the set of gradients, we are assessing the complexity of the gradient in terms of covering numbers. Using Sudakov inequality (see \cite[Theorem 3.18]{LT}), it is possible to compare \cite[Corollary 2]{Eldan} with Theorem \ref{partitionfunc}. More precisely, we have the following result as a corollary of Theorem \ref{partitionfunc}.   

\begin{Cor}\label{GW}
Assume $\mu$ is compactly supported. Denote by $K$ the convex hull of its support and assume that the diameter of $K$ is $\sqrt{n}$. Let $f : \RR^n \to \RR$ be a continuously differentiable function and let $V = \nabla f(K)$. There exists a numerical constant $\kappa>0$ such that,
$$\log \int e^{f} d\mu \leq \sup_{y \in \RR^n} \{ f(y) -\Lambda^*_\mu(y)\}  +\kappa n^{\frac{1}{3}} g(V)^{\frac{2}{3}},$$
where $g(V)$ is the Gaussian mean-width of $V$, defined as
$$ g(V) = \EE \sup_{\zeta \in V} \langle \zeta , \Gamma \rangle,$$
with $\Gamma$ being a standard Gaussian random variable in $\RR^n$.
\end{Cor}
\begin{proof}
From Theorem \ref{partitionfunc} we have for any $\delta>0$,
$$\log \int e^{f} d\mu \leq \sup_{y \in \RR^n}  \{ f(y) -\Lambda^*_\mu(y)\}  + \log N\big(\mathrm{co}(V), (\delta/\sqrt{n}) B_{\ell^2}\big) +  \delta,$$
where $\mathrm{co}(V)$ denotes the convex hull of $V$ and $N\big(\mathrm{co}(V), (\delta/\sqrt{n}) B_{\ell^2}\big)$ the covering number of $\mathrm{co}(V)$ by $(\delta/\sqrt{n}) B_{\ell^2}$. By Sudakov inequality (see \cite[Theorem 3.18]{LT}), we know that there exists a numerical constant $c>0$ such that
$$  \log N(\mathrm{co}(V), (\delta/\sqrt{n}) B_{\ell^2})  \leq c \frac{n g(\mathrm{co}(V))^2}{\delta^2}.$$
But as the supremum of a convex function on a set is equal to the supremum on its convex hull, we have $g(\mathrm{co}(V)) = g(V)$. Thus,
$$\log \int e^{f} d\mu \leq \sup_{y \in \RR^n}  \{ f(y) -\Lambda^*_\mu(y)\}  +  c \frac{n g(V)^2}{\delta^2} +  \delta.$$
Optimizing in $\delta>0$, we obtain that there exists a numerical constant $\kappa>0$ such that
$$\log \int e^{f} d\mu \leq \sup_{y \in \RR^n}  \{ f(y) -\Lambda^*_\mu(y)\}  + \kappa n^{\frac{1}{3}} g(V)^{\frac{2}{3}}.$$
\end{proof}
As for the lower bound, it is a classical fact to note that as a consequence of Jensen's inequality,  one always has the following a lower bound by just changing the barycenter of the underlying reference measure. 
\begin{Lem}\label{lbpartition}Let $f : \RR^n\to \RR$ be a non-negative measurable function. Then,
$$ \log \int e^f d\mu \geq \sup \big\{ \int f d\mu_{y} - \Lambda^*_{\mu}(y), \ y \in \nabla \Lambda_\mu \big(\mathrm{int}(\mathcal{D}_{\Lambda_\mu}) \big) \big\},$$
where $\mu_y$ is defined in \eqref{defmuh}.
\end{Lem}
\begin{Rem}\label{gappartition}
Note the gap between the lower bound of Lemma \ref{lbpartition} and the main term $\sup\{ f -\Lambda^*_\mu\}$ in the upper bound of Proposition \ref{partitionfunc}. 

When $\mu$ is the uniform measure on the discrete hypercube, one can consider, given a function $f : \{0,1\}^n \to \RR$, its natural harmonic extension, and the just mentioned gap resolves itself. One can view this fact as at the heart of Eldan's approach \cite{Eldan}.
For further discussion on the general case, we refer the reader to Remark \ref{gap1}.

\end{Rem}
We will now state our nonlinear large deviations upper bound.
We recall from \cite[Section 1.2]{DZ} that a sequence of random variables $(Z_n)_{n\in \NN}$ taking value in some topological space $\mathcal{X}$ equipped with the Borel $\sigma$-field $\mathcal{B}$, satisfies a \textit{large deviations principle} (LDP) with \textit{speed} $\upsilon_n$, and \textit{rate function} $I : \mathcal{X} \to [0, +\infty]$, if $I$ is lower semicontinuous, $\upsilon_n$ increases to infinity, and for all $B\in \mathcal{B}$,
\begin{align}
- \inf_{\mathrm{int}(B)}I \leq \liminf_{n\to +\infty} \frac{1}{\upsilon_n}& \log \PP\left(Z_n \in B\right)\nonumber\\
& \leq \limsup_{n\to +\infty} \frac{1}{\upsilon_n}\log\PP\left(Z_n \in B\right) 
\leq -\inf_{\mathrm{cl}(B)  } I,\label{upperboundLDP}\end{align}
where $\mathrm{int}(B)$ denotes the interior of $B$ and $\mathrm{cl}(B)$ the closure of $B$. We recall that $I$ is \textit{lower semicontinuous} if its $t$-level sets $\{ x \in \mathcal{X} : I(x) \leq t \}$ are closed, for any $t\in [0,+\infty)$. Furthermore, if all the level sets are compact, then $I$ is said to be a \textit{good rate function}.

By \cite[(1.2.7)]{DZ}, for good rate functions $I$, the large deviations upper bound, that is the right-most inequality in \eqref{upperboundLDP}, is equivalent to the statement that for any $r>0$ and $\delta>0$, 
$$ \limsup_{n\to +\infty} \frac{1}{v_n} \log \PP\big( Z_n \notin V_{\delta}(\{I\leq r\}) \big) \leq - r,$$
where for any subset $A \subset \RR$,  $V_{\delta}(A)$ denotes the \textit{open $\delta$-neighborhood of $A$}, that is,
$$ V_{\delta}(A) = \big\{ x \in \RR : \inf_{y \in A} |x-y|< \delta\big\}.$$
With that in mind, the following proposition is our general large deviation upper bound.
\begin{The}\label{NL}
Let $X$ be a random vector in $\RR^n$ with law $\mu$ such that $\Lambda_\mu$ is essentially smooth. Let $f : \RR^n \to \RR$ be a measurable function. Define,
$$\forall t \in \RR, \  I(t) = \inf \{ \Lambda^*_{\mu}(y) : f(y)= t, y \in \RR^n \}.$$
Let $r>0$. Assume there exists $\kappa\geq r$ such that denoting by $K = \{ \Lambda^*_\mu \leq \kappa\}$,
\begin{equation} \label{tightness}\PP(X\notin  K)\leq e^{-r}.\end{equation}
Let $W \subset \RR^n$ be a compact convex set such that, 
\begin{equation} \label{meanvalueass} \forall x,y \in K, \ f(x)-f(y) \leq \sup_{\lambda \in W} \langle \lambda, x-y\rangle,\end{equation}
Denote by $D$ the diameter of $K$. Let $\delta>0$  and let $\mathcal{D}_{\delta/2}$ be a $\delta/2D$-net of $W$ for the $\ell^2$-norm.
 Then,
$$\log \PP\big( f(X) \notin V_{\delta}(\{I\leq r\}) \big) \leq -r +\log|\mathcal{D}_{\delta/2}|+ \log\Big(\Big( \frac{\kappa LD }{\delta}\Big)\vee 1\Big)  + 2,$$
where $L = \sup_{\lambda \in V} ||\lambda ||_{\ell^2}$.
\end{The}
\begin{Rem} \label{smooth}

The wording of Theorem \ref{NL} is intended to include non-smooth functions which we will encounter in the applications. It is particularly relevant for  functions $f$ which are linear combinations of non-smooth convex functions, for which we can take $W$ to be  Minkowski sums of sets of subdifferentials of the convex functions involved in the decomposition.

\end{Rem}
\begin{Rem}\label{prod}
The tightness assumption \eqref{tightness} in the Proposition \ref{NL} is automatically satisfied for product measures. Indeed, 
we know from \cite[Lemma 5.1.14]{DZ} that if $\mu$ is a probability measure on $\RR$, then for any $\alpha \in(0,1)$,
$$ \EE e^{\alpha \Lambda^*_{\mu}(X_1)} \leq \frac{2}{1-\alpha}.$$
Taking $\alpha = 1/2$, we deduce that for $\kappa = 12(r\wedge n)$, Chernoff's inequality gives
$$ \PP( \Lambda^*_{\mu^n}(X) > \kappa ) \leq e^{-r}.$$
\end{Rem}

As a consequence of Theorem \ref{NL}, we obtain a nonlinear large deviations estimate for the upper tail of a function with a low-complexity gradient.
\begin{Cor}\label{CorNL}
Let $X$ be a random vector in $\RR^n$ with law $\mu$ such that $\Lambda_\mu$ is essentially smooth. Let $f : \RR^n \to \RR$ be a measurable function. Define,
$$\forall t \in \RR, \  \phi(t) = \inf \{ \Lambda^*_{\mu}(y) : f(y)\geq  t, y \in \RR^n \}.$$
Let $t,\delta>0$ and assume that $\phi$ is increasing on $[t-\delta,+\infty)$. 
Let $\kappa\geq \phi(t-\delta)$ such that denoting by $K = \{ \Lambda^*_\mu \leq \kappa\}$,
\begin{equation} \label{tightness}\PP(X\notin  K)\leq e^{-\phi(t-\delta)}.\end{equation}
Let $W \subset \RR^n$ be a compact convex set such that, 
$$ \forall x,y \in K, \ f(x)-f(y) \leq \sup_{\lambda \in W} \langle \lambda, x-y\rangle,$$
Denote by $D$ the diameter of $K$ and let $\mathcal{D}_{\delta/2}$ be a $\delta/2D$-net of $W$ for the $\ell^2$-norm.
 Then, 
$$\log \PP\big( f(X) \geq t \big) \leq -\phi(t-\delta) +\log|\mathcal{D}_{\delta/2}|+ \log\Big(\Big( \frac{\kappa LD }{\delta}\Big)\vee 1\Big)  + 2,$$
where $L = \sup_{\lambda \in V} ||\lambda ||_{\ell^2}$.
\end{Cor}

\begin{proof}
We claim that our assumption on the strict monotonicty of $\phi$ yields that 
\begin{equation} \label{inclusion} \{ I\leq \phi(t-\delta)\}\subset (-\infty, t-\delta].\end{equation}
Indeed, for any $z \in \RR$, if $I(z) \leq \phi(t-\delta)$, then in particular $\phi(z)\leq \phi(t-\delta)$. Since $\phi$ is increasing on $[t-\delta,+\infty)$, we have $z \leq t-\delta$, which shows \eqref{inclusion}.
Therefore, if $x\in \RR^n$ is such that $f(x)\geq t$ then  $f(x)\notin V_\delta(\{ I\leq \phi(t-\delta)\})$. Thus, we have
$$ \PP( f(X) \geq t) \leq \PP\big( f(X) \notin V_\delta(\{ I\leq \phi(t-\delta)\})\big),$$
which, by Theorem \ref{NL}, ends the proof.
\end{proof}

Turning our attention to the lower bound, one can always get a large deviations lower bound by finding a tilting of the underlying measure $\mu$ which transforms a large deviations event for $\mu$ into a typical event for $\mu_h$. (This idea is prevalent in the large deviations literature, and is also the basis for the proof of Lemma \ref{lbpartition}). The following lower bound, which builds on  \cite[p. 64]{Ledouxflour} quantifies this idea. First, note that differentiating twice $\Lambda_\mu$, one obtains the covariance of the measure $\mu_y$, that is,
$$\forall \eta \in \RR^n, \  \Var_{\mu_y} \langle \eta, X \rangle =\langle \eta,  \nabla^2 \Lambda_\mu(\lambda).\eta\rangle ,$$
where $y = \nabla \Lambda_\mu(\lambda)$.  With this notation, one obtains the following lemma.
\begin{Lem}\label{borneinfdet}
 Let $V \subset \RR^n$ be a measurable subset, and $\lambda \in  \mathrm{int}(\mathcal{D}_{\Lambda_{\mu}})$.  Then,
$$\mu(V) \geq e^{-\Lambda_{\mu}^*(y)} \mu_{y}(V)\exp\Big(-\frac{1}{\mu_{y}(V)^{1/2}} \langle \lambda, \nabla^2\Lambda_{\mu}(\lambda).\lambda \rangle^{1/2} \Big),$$
where $ y= \nabla \Lambda_{\mu} (\lambda) $, and $\mu_y$ is defined in \eqref{defmuh}.
\end{Lem}

\begin{proof}Let $\lambda \in \mathrm{int}(\mathcal{D}_{\Lambda_{\mu}})$ and $y = \nabla \Lambda_\mu(\lambda)$. 
Changing the measure $\mu$ into $\mu_y$, we have
$$\mu(V) = \int_{V} e^{-(\langle \lambda , x\rangle - \Lambda_\mu(\lambda))} d\mu_{y} (x) = e^{-\Lambda_{\mu}^*(y)}\int_{V} e^{-\langle \lambda ,  x-y\rangle}d\mu_{y}(x),$$
where we used the fact that $\Lambda^*_\mu(y) = \langle \lambda,y\rangle - \Lambda_\mu(\lambda)$.
By Jensen's inequality we deduce,
$$\mu(V) \geq e^{-\Lambda_{\mu}^*(y)}\mu_{h}(V) \exp\Big( -\frac{1}{\mu_{y}(V)} \int_{V} \langle \lambda , x-y \rangle d\mu_{y}(x)\Big).$$
Using Cauchy-Schwarz inequality we get,
$$ \int_{V} \langle \lambda , x-y \rangle d\mu_{y}(x) \leq \mu_{y}(V)^{1/2} \Big(\int \langle \lambda , x-y \rangle^2 d\mu_{y}(x)\Big)^{1/2},$$
which yields the claim.

\end{proof}

\subsection{Applications}
In this section, we provide several examples in which the gap between the upper bounds (Propositions \ref{partitionfunc} and \ref{NL}) and lower bounds (Lemmas \ref{lbpartition} and \ref{borneinfdet}) is negligible in the large deviation scale. 
\subsubsection{Mean-field approximation in the Ising model}\label{Ising}
Our first example of application is the accuracy of the mean-field prediction in the Ising model with large degree. This topic is discussed in \cite{BM},  \cite[section 1.3]{Eldan}, \cite{JKM}, \cite{JKR}, and we will improve on some of the results therein.

 In particular, we show that the mean-field approximation of the Ising model holds true as soon as the empirical distributions of the eigenvalues of the interaction matrices converge weakly to a Dirac at $0$, and the second moments are uniformly bounded. 

Let $\mathcal{H}_n$ denote the set of Hermitian matrices of size $n$. For any $A\in\mathcal{H}_n$, we denote by $\mu_A$ the empirical distribution of its eigenvalues, defined by,
$$   \mu_A = \frac{1}{n} \sum_{i=1}^n \delta_{\lambda_i},$$
where $\lambda_1,\ldots ,\lambda_n$ are the eigenvalues of $A$.
\begin{Pro}\label{meanfield}
Let $A_n$ be a sequence of Hermitian matrices such that for any $i\in\{1,\ldots ,n\}$,  $(A_n)_{i,i} =0$. Assume that the empirical distribution of eigenvalues $\mu_{A_n}$ converges weakly to a Dirac at $0$ and that  $\limsup_n n^{-1} \tr A_n^2 <+\infty$. 
Let $\mu$ be the uniform probability  measure on $\{-1,1\}^n$. Then,
$$ \log \int e^{\langle \sigma , A_n\sigma \rangle} d \mu(\sigma)= \sup_{ x \in [-1,1]^n} \big\{ \langle x, A_n x\rangle - \Lambda^*_{\mu}(x)\big\} + o(n).$$

\end{Pro}
\begin{proof}Note that the lower bound 
$$ \log \int e^{\langle \sigma , A_n\sigma \rangle}d\mu(\sigma) \geq \sup_{ x \in [-1,1]^n} \big\{ \langle x, A_n x\rangle - \Lambda^*_{\mu}(x)\big\}, 
$$ 
follows directly  Lemma \ref{lbpartition}.
Let 
$$ \forall x \in \RR^n, \ f(x) = \langle x, A_n x\rangle.$$
The following lemma taken from \cite[Lemma 3.4, Remark 4.1]{BM} computes the complexity of the gradient of $f$ under the mean-field assumption we made.

\begin{Lem}[  {\cite[Lemma 3.4, Remark 4.1]{BM}} ]\label{complexIsing}
Under the assumption of Proposition \ref{meanfield}, for any $\delta>0$, there exists a $\delta\sqrt{n}$-net $\mathcal{E}_{\delta\sqrt{n}}$ for the $\ell^2$-norm of the set $A_n([-1,1]^n)$, such that,
$$ \log |\mathcal{E}_{\delta\sqrt{n}}| = o(n).$$
\end{Lem}

As $\nabla f(x) = 2A_nx$ for any $x\in[-1,1]^n$, applying Proposition \ref{partitionfunc} and using the above lemma, we deduce that for any $\delta>0$, 
\begin{equation} \label{Isingbound}\log \int e^{\langle \sigma, A_n \sigma\rangle} d\mu(\sigma) \leq \sup_{x\in [-1,1]^n}\{ \langle x, A_n x\rangle -\Lambda^*_{\mu}(x) \} + \log |\mathcal{E}_{\delta\sqrt{n}}| + 4\delta n,\end{equation}
which gives the claim.
\end{proof}

\begin{Rem}\label{remRonen}
Instead of using Lemma \ref{complexIsing} to bound the error term in \eqref{Isingbound}, one can use alternatively Corollary \ref{GW}, which gives, 
$$ \log \int e^{\langle \sigma, A_n \sigma\rangle} d\mu(\sigma) \leq \sup_{x\in [-1,1]^n}\{ \langle x, A_n x\rangle -\Lambda^*_{\mu}(x) \} +   \kappa n^{\frac{1}{3}} \big(\EE \sup_{x \in [-1,1]^n } \langle A_n x,\Gamma \rangle\big)^{\frac{2}{3}},$$
where $\kappa>0$ and $\Gamma$ is a standard Gaussian vector in $\RR^n$. But, using twice Cauchy-Schwarz inequality and the fact that $\EE || A_n \Gamma||^2 = ||A_n||_2$, one obtains
$$  \EE \sup_{x \in [-1,1]^n } \langle A_n x,\Gamma \rangle \leq \sqrt{n}||A_n||_2,$$
where $|| \ ||_2$ denotes the Hilbert-Schmidt norm. Thus, 
$$ \log \int e^{\langle \sigma, A_n \sigma\rangle} d\mu(\sigma) \leq \sup_{x\in [-1,1]^n}\{ \langle x, A_n x\rangle -\Lambda^*_{\mu}(x) \} + \kappa(||A_n||_2 n)^{\frac{2}{3}},$$
 thus giving another proof of the result of Jain, Koehler and Risteski \cite{JKR}.

\end{Rem}

\subsubsection{Traces of powers of Wigner matrices}

In this section, we discuss the large deviations of the normalized traces of powers of Wigner matrices with sub-Gaussian entries. Under technical assumptions, we show that we can find large deviation upper and lower bounds, which match in the case where the entries have \textit{sharp sub-Gaussian tails} and have the same covariance structure as the GOE or the GUE. In the latter case, the large deviations are universal, that is, the resulting rate function is the same for all such Wigner matrices, and coincides with the Gaussian case.   
This universality phenomenon was first discovered by Guionnet and Husson in \cite{GH}, in the context of the large deviations of the largest eigenvalue of Wigner matrices.

To state our result we need to introduce some notations. Denote by $\mathcal{H}_n^{(\beta)}$ the set of Hermitian matrices when $\beta=2 $, and symmetric matrices when $\beta=1$, of size $n$, which we will short-hand as $\mathcal{H}_n$ whenever there is no ambiguity.
Say that $X$ is a \textit{Wigner matrix} if $X$ is a random Hermitian matrix with independent coefficients (up to the symmetry) such that both $(X_{i,i})_{1\leq i \leq n}$ and $(X_{i,j})_{i<j}$ are identically distributed,  $\EE X=0$, $(\Re X_{1,2}, \Im X_{1,2})$ are independent and $\EE|X_{1,2}|^2 =1$.  By Wigner's theorem (see \cite[Theorem 2.1.1, Exercice 2.1.16]{AGZ}, \cite[Theorem 2.5]{Silverstein}), we know that
$$\mu_{X/\sqrt{n}} \underset{n\to\infty}{\leadsto} \mu_{sc},$$
in probability, where $\leadsto$ denotes the weak convergence, and $\mu_{sc}$ is the semi-circular law defined by,
$$ \mu_{sc} = \frac{1}{2\pi} \sqrt{4-x^2} \Car_{|x| \leq 2} dx.$$
If we assume moreover that for any $d\in \NN$,
$$ \max\left( \EE|X_{1,1}|^d,\EE|X_{1,2}|^d \right) < +\infty,$$
then we have by \cite[Lemmas 2.1.6,2.1.7]{AGZ}, the convergence of the moments of $\mu_{X/\sqrt{n}}$ towards the ones of the semi-circular law, that is for any $d\in \NN$,
$$\frac{1}{n} \tr (X/\sqrt{n})^d  \underset{n\to+\infty}{\longrightarrow }  \mu_{sc}(x^d),$$
in probability, where $\mu_{sc}(x^d) = \int x^d d\mu_{sc}(x)$. It is known that
$$ \mu_{sc}(x^d) = \begin{cases}
C_{d/2}&\text{ if $d$ is even,}\\
0 & \text{ if $d$ is odd,}
\end{cases} 
$$
where in the case $d$ is even, $C_{d/2} = \frac{1}{d/2+1} \binom{d}{d/2}$ is the $\left(\frac{d}{2}\right)^{\text{th}}$ Catalan number. The question we investigate is the one of the large deviations of the these traces around the moments of the semi-circular law.

We will actually prove a large deviation result under a more restrictive assumption than just the one of having sub-Gaussian entries, as we will need to use further concentration arguments in addition to the result of Theorem \ref{NL}. To this end, we introduce a class of random Hermitian matrices satisfying a certain \textit{convex concentration property}.

\begin{Def}\label{concconv}
We say that a random Hermitian matrix $X$ satisfies the \textit{convex concentration property} if there exists $c>0$ such that if for any $f : \mathcal{H}_n \to \RR$ convex $1$-Lipschitz function with respect to the Hilbert-Schmidt norm, $\EE f(X)$ exists and for any $t>0$,
$$ \PP\big( |f(X) -\EE f(X)| >t \big) \leq c^{-1}\exp\big(- c t^2\big).$$
\end{Def}

\begin{Rem}\label{remconc}
A random Hermitian matrix whose law satisfies a log-Sobolev inequality has normal concentration by \cite[Theorem 5.3]{Ledouxmono}. In particular, such random matrices satisfy the above convex concentration property. Workable criteria for log-concave measures to satisfy a log-Sobolev inequality include the strict uniform convexity of the potential, see \cite[Theorem 5.2]{Ledouxmono}. Due to the tensorization property of the log-Sobolev inequality  (see \cite[Corollary 5.7]{Ledouxmono}), any Wigner matrix with entries satisfying a log-Sobolev inequality has the convex concentration property.

 Another important family of Wigner matrices satisfying this property are the one with bounded entries, due either to Talagrand's inequality \cite[Corollary 4.10]{Ledouxmono} or due to a result by the transportation method \cite[Theorem 8.6]{BLM} (the latter having the advantage of directly measuring deviations from the mean).

\end{Rem}
In the following, we denote by $\Lambda$ the logarithmic Laplace transform of the law of $X$, that is,
$$\forall H\in\mathcal{H}_n, \ \Lambda(H) = \log \EE e^{\tr (XH)},$$
and by $\Lambda_{i,j}$ the one of the law of each of its entries $X_{i,j}$. Finally, we denote by $\Lambda^*$ the Legendre transform of $\Lambda$.

Using Proposition \ref{NL}, together with a truncation argument which enable us to reduce ourself to a function with a low complexity gradient, we obtain the following large deviation result for traces of Wigner matrices.

\begin{The}\label{LDPtrb}Let $d \geq 3$, and $\ell>d$ be an even integer.
Let $X$ be a Wigner matrix satisfying the convex concentration property such that $\Lambda_{1,1}$ and $\Lambda_{1,2}$ have their derivatives of order $2$ to $\ell$ uniformly bounded.  For any closed subset $F$ of $\RR$,
$$\limsup_{n\to +\infty} \frac{1}{n^{1+\frac{2}{d}}} \log \PP \Big( \frac{1}{n} \tr(X/\sqrt{n})^d \in F \Big) \leq - \inf_F I_+,$$
and for any open subset $O$ of $\RR$,
$$\liminf_{n\to +\infty} \frac{1}{n^{1+\frac{2}{d}}} \log \PP \Big( \frac{1}{n} \tr(X/\sqrt{n})^d \in O \Big) \geq - \inf_O I_-,$$
where for any $x\in\RR$,
$$ I_+(x) =  \sup_{\delta>0}\liminf_{n\to +\infty} I_{n,\delta}(x),$$
$$ I_-(x) =\sup_{\delta>0}\limsup_{n\to +\infty} I_{n,\delta}(x),$$
and for any $n\in \NN$, $\delta>0$,
$$ I_{n,\delta}(x) = \inf \Big\{ \frac{\Lambda^*(Y)}{n^{1+\frac{2}{d}}} : |\frac{1}{n} \tr (Y/\sqrt{n})^d + \mu_{sc}(x^d) - x| <\delta, Y \in \mathcal{H}_n \Big\}.$$ 
\end{The}
\begin{Rem}
Together with the previous Remark \ref{remconc}, we see that any Wigner matrix with bounded entries satisfies the assumptions of Theorem \ref{LDPtrb}.

\end{Rem}

We will now specify our result to Wigner matrices with sharp sub-Gaussian tails. Following Guionnet and Husson \cite{GH}, we say that a real or complex random variable $\xi$ has \textit{sharp sub-Gaussian tails} if,
$$ \forall z \in \CC, \  \Lambda_{\xi}(z) = \log  \EE e^{\Re(z \xi)} \leq \frac{1}{2}\langle z, \Sigma z\rangle,$$
where $\Sigma$ is the covariance matrix of $(\Re \xi, \Im \xi)$, and $\langle . , . \rangle$ is the standard inner product in $\CC$.

\begin{Cor}\label{sharpsubGtr}
Let $d\in \NN, d\geq 3$.  Let $X$ be a Wigner matrix such that $\EE X_{1,1}^2 \leq 2/\beta$. Assume that $X$ is real symmetric if $\beta =1$, and if $\beta=2$, that $(\Re X_{1,2}, \Im X_{1,2})$ are independent each of variance $1/2$. In addition to the assumptions of Theorem \ref{LDPtrb}, assume that the entries of $X$ have sharp sub-Gaussian tails. 
The sequence $(\frac{1}{n}\tr(X/\sqrt{n})^d)_{N\in \NN}$
satisfies a LDP with speed $n^{1+\frac{2}{d}}$, and good rate function $J_d$.
If $d$ is even, $J_d$ is given by,
$$\forall x \in \RR, \ J_d(x) = \begin{cases}
\frac{\beta}{4} \left( x-C_{d/2}\right)^{\frac{2}{d}}& \text{ if } x\geq C_{d/2},\\
+\infty & \text{ otherwise,}
\end{cases}
$$
where $C_{d/2}$ denotes the $\left( \frac{d}{2}\right)^{\text{ th}}$ Catalan number, and if $d$ is odd, 
$$\forall x \in \RR, \ J_d(x) = \frac{\beta}{4}
\left| x\right|^{\frac{2}{d}},$$
where $\beta =1$ if $X$ is real symmetric and $\beta=2$ if $X$ is complex Hermitian.
\end{Cor}

\begin{Rem}
 Distributions which have sharp sub-Gaussian tails include the Rademacher distribution $\frac{1}{2} \delta_{1} + \frac{1}{2} \delta_{-1}$ and uniform probability measures on intervals symmetric around the $0$, see \cite[Examples 1.2]{GH}.
Therefore, the above LDP holds in particular for Wigner matrices with Rademacher entries, or uniformly sampled in $[-\sqrt{3},\sqrt{3}]$.

\end{Rem}
\subsubsection{Upper tail of cycle counts in sparse  Erd\H{o}s--R\'enyi graphs}
 Let $X$ be the adjacency matrix of an  Erd\H{o}s--R\'enyi graph on $n$ vertices with parameter $p_n=p$. More precisely, we assume that $X$ is a symmetric matrix, such that  $(X_{i,j})_{i<j}$ are independent Bernoulli random variables of parameter $p$, whereas $X_{i,i}=0$ for any $i\in \{1,\ldots,n\}$.

For a given finite graph $H$, denote by $X_H$ the number of copies of $H$ in the  Erd\H{o}s--R\'enyi graph. A general question is to understand the large deviation of $X_H$ of the order of its expectation when $n$ goes to $+\infty$. When $p$ is fixed, the large deviation of these sub-graph counts are well understood now due to the work of Chatterjee and Varadhan \cite{CV} on the large deviations of the  Erd\H{o}s--R\'enyi graph for the cut-metric. 

When $p\ll 1$, a first question is to estimate the order of the upper tail, that is, for $u \geq 1$,
 $$ \PP( X_H \geq u \EE X_H).$$
In the case of triangles $H = K_3$, it was proven in a series of papers \cite{Vugraph}, \cite{KV}, \cite{Chatterjeetriangle}, \cite{DK1}, that for any $p\geq 1/n$,
$$ \log \PP( X_{K_3} \geq u \EE X_{K_3})  \asymp  -\min(n^2 p^{2}\log(1/p), n^3p^3),$$
and then generalized in \cite{DK} for cliques of arbitrary size. For general graphs $H$, the order of the upper tail has been computed up to some $\log (1/p)$ factor, by Janson, Oleszkiewicz and Ruci\'{n}ski  in \cite{JOR}. In particular, denoting by $\Delta$ is the highest degree in $H$, they proved in \cite[Theorems 1.2,1.5]{JOR}, that 
$$ -n^2 p^{\Delta}\log(1/p) \lesssim  \log \PP( X_{H} \geq u \EE X_{H})  \lesssim  -n^2 p^{\Delta},$$
for $p\geq n^{-1/\Delta}$.
Still, the exact order of this tail is not fully understood for small edge-probability, as the order conjectured in \cite{DK} by Kahn and DeMarco was recently disproved by \v Sileikis and Warnke in \cite{SW}.


Working instead with homomorphism densities, that is, if $H$ has vertex set $V$ and edge set $E$, 
$$ t(H,X) = \frac{1}{n^{|V|}} \sum_{ \underline{i} \in \{1,...,n\}^V} \prod_{(v,w) \in E} X_{\underline{i}(v), \underline{i}(w)},$$
Chatterjee and Dembo showed in \cite{CD} that the large deviations of $t(H,X)$ fall into their framework of nonlinear large deviations. The large deviation upper tail is understood by a certain variational problem,
\begin{equation} \label{equiv}  -\log \PP\big( t(H,X) \geq u \EE t(H,X) \big)  \sim \phi_{n,H}(u),\end{equation}
with
\begin{equation} \label{pbvar} \phi_{n,H}(u) = \inf \big\{ \Lambda^*_p(Y) : t(H,Y) \geq u \EE t(H,Y),  \ Y\in \mathcal{H}_n^0 \big\},\end{equation}
where $\mathcal{H}_n^0$ is the set of symmetric matrices of size $n$ with null diagonal coefficients, and  $\Lambda^*_p(Y) = \sum_{i<j} I_p(Y_{i,j})$, 
where $$\forall x \in [0,1], \ I_p(x) = x \log \frac{x}{p} +(1-x)\log \frac{1-x}{1-p},$$
 and $+\infty$ otherwise. The equivalent \eqref{equiv} was shown to hold in \cite[Theorem 1.2]{CD} for $p\geq n^{-\alpha(H)}$, for some explicit constant $\alpha(H) >0$ depending on $H$. In the case where $H=K_3$, the threshold of validity of \eqref{equiv} found was $n^{-1/42}$, up to logarithmic factor, and was pushed later on by Eldan in \cite{Eldan} to $n^{-1/18}$. Yet, neither of these thresholds are expected to be optimal, and in this case of triangle counts, it is conjectured that \eqref{equiv} holds as soon as the variational problem \eqref{pbvar} gives the right order of the upper tail, that is, with the help of \cite{LZ}, $n p\gg \log n$. 

Very recently, Cook and Dembo \cite{CoD} proved that the nonlinear large deviations of $t(H,X)$ holds for the range $n^{-1/(5\Delta -4)} \ll p \ll 1$, and more strongly, in the case of $d$-cycles counts for $n^{-1/2}  \log n \ll p \ll 1$ when $d\geq 4$, and $ n^{-1/3}   \ll p \ll 1$ when $d=3$. 

We will give an alternative proof of their result on cycle counts. In the following proposition, we push the estimation of the upper tail for sparsity parameters satisfying $n^{-1/2} \log^4 n \ll p \ll 1$, thus improving on Cook and Dembo's result in the case $d=3$.

\begin{Pro}\label{cycleup}
Let $p$ such that $p=o(1)$ and $\log^4n =o(np^2)$. Denote by  $v_n = n^2p^2 \log(1/p)$.
For any $t\geq 1$,
$$ \log\PP\big( \tr( X^d )\geq t n^dp^d \big) \leq -  \phi_n(t) + o(v_n),$$
where $$ \phi_n(t) = \inf \big\{ \Lambda^*_p(Y)  : \tr (Y^d) \geq t (np)^d, Y\in\mathcal{H}_n^0 \big\}.$$ 
\end{Pro}
The proof, as for the traces of Wigner matrices, rely on a truncation argument which enables us to  lower the complexity of our function of interest and apply efficiently Proposition \ref{NL}. 

In \cite[Theorem 1.5]{BGLZ} (or \cite[Theorem 1.1]{LZ} for the case $d=3$) the variational problem was solved asymptotically.  They showed that for $np\gg 1$,
\begin{equation} \label{defPhic} \frac{\phi_n(t)}{v_n}  \underset{n\to +\infty}{\longrightarrow} \Phi(t),\end{equation} where $\Phi$ is given by, 
\begin{equation} \label{defPhi} \Phi(t) = \begin{cases}
\min\big( \theta_t, \frac{1}{2}(t-1)^{\frac{2}{d}}\big) &\text{ if } p \gg n^{-1/2},\\
\frac{1}{2}(t-1)^{\frac{2}{d}} & \text{ if } n^{-1} \ll p \ll n^{-1/2},
\end{cases}\end{equation}
with $\theta_t$ the unique solution of the equation $P_{C_d}(\theta_t) = t$, where $P_{C_d}$ is the independence polynomial of the $d$-cycle.
The optimizers correspond on one hand to planting a clique of order $np (t-1)^{1/d}$, which gives the $\frac{1}{2} (t-1)^{2/d}$ term, and on the other hand to planting an anti-clique of order  $np^2 \theta_{t}$ which corresponds to the other $\theta_t$ term. 

With this knowledge of the optimizers, one can obtain the complementary lower bound, so that together with Proposition \ref{cycleup}, we get, for any $t\geq 1$,
$$ \lim_{n\to +\infty} \frac{1}{v_n} \log\PP\big( \tr (X^d) \geq t n^dp^d \big) =-  \Phi(t),$$
for any $n^{-1/2}\log^4 n \ll p\ll 1$.  The edge-probability $n^{-1/2}$ appears to be critical for cycles in different ways. Note that below this threshold, the anti-clique construction is no longer available. Moreover, for this parameter,  the order at the exponential scale of the upper tail approaches $n$. From the point of view of the truncation method we are using, the speed $n$ arises as being impassable (without further dimension reduction): the complexity of the gradient of a truncated trace is at least the one of the $(n-1)$-dimensional sphere, as it corresponds to pick the top eigenvector.

 Furthermore, the threshold $n^{-1/2}$ is also critical in the sense that one observes a change of speed in the large deviation lower tail of cycles counts. Indeed, we know by the work of  Janson,  {\L}uczak and Ruci{\'n}sk \cite{JLR}, and Janson and Warnke \cite{JW}, that in the regime where $p=o(1)$, for any finite graph $H$, and any $0< u \leq 1$,
$$ \log \PP( X_H \leq u \EE X_H) \asymp -(1-u)^2 \Phi_H,$$
where $\Phi_H = \min \{ \EE X_J : J \text{ sub-graph of } H \}$. Thus, the speed of the lower tail is driven by the ``least expected'' sub-graph of $H$. For example for $H = K_3$, this implies that,
$$ \log \PP( X_H \leq (1-u) \EE X_H) \asymp -u^2 \min(n^3p^3, n^2p).$$
Thus a change of speed happens at $n^{-1/2}$, and similarly for cycles of greater lengths.

In \cite[Theorem 1.2]{CoD}, the authors provided sharp lower tail estimates of homomorphisms densities of $d$-cycles for sparsity parameters $p\gg n^{-c}$ (up to logarithmic corrections), with $c=\frac{d-2}{2(d-1)}$.    We mention also that an entropic perspective on the estimation of the lower tail of triangle counts has been announced by Kozma and Samotij \cite{KS} all the way down to  $p\gg n^{-1/2}$. 

\section{Nonlinear large deviation upper bounds}\label{upper}
We give here a proof of the two Theorems \ref{partitionfunc} and \ref{NL}. 
\begin{proof}[Proof of Proposition \ref{partitionfunc}]
Let $x,y\in K$. We have by the mean-value theorem,
\begin{align*} f(x) &= f(x) -f(y) + \Lambda^*_{\mu}(y) + f(y) - \Lambda^*_{\mu}(y)\\
& \leq \sup_{t\in K} \langle \nabla f(t),x-y\rangle +\Lambda^*_\mu(y) +\sup_{z \in \RR^n} \{ f(z) -\Lambda^*_\mu(z) \}.
\end{align*}
Denote by $W$ the convex hull of $\nabla f(K)$ and $h_W$ its support function, that is, $h_W(y) = \sup_{\lambda \in W} \langle \lambda , y\rangle$ for any $y\in \RR^n$. With this notation, we get,
$$ f(x) \leq h_W(x-y) +\Lambda^*_\mu(y) + \sup_{z \in \RR^n} \{ f(z) -\Lambda^*_\mu(z) \}.$$
Optimizing in $y$, we deduce that
$$ f(x) \leq  \inf_{ y\in K} \{ h_W(x-y) +\Lambda^*_\mu(y) \}+ \sup_{z \in \RR^n} \{ f(z) -\Lambda^*_\mu(z) \},$$
Using the minimax theorem (see \cite[Theorem 4.36]{Clarke}), we have
\begin{align*}
 \inf_{ y\in K} \{ h_W(x-y) +\Lambda^*_\mu(y) \} &= \sup_{\lambda \in W} \inf_{y\in K} \{\langle \lambda, x-y\rangle +\Lambda^*_\mu(y) \}\\
& =  \sup_{\lambda \in W} \{ \langle \lambda, x\rangle -\Lambda_\mu(\lambda)\},
\end{align*}
where we used the fact that $\Lambda^*_\mu(y) = +\infty$ if $y \notin K$ by the Hahn-Banach theorem, and that $\Lambda_\mu$ is the Legendre transform of $\Lambda^*_{\mu}$ by \cite[Theorem 4.21]{Clarke}.
Therefore,
$$f(x) \leq \sup_{\lambda \in W} \{\langle \lambda, x\rangle - \Lambda_\mu(\lambda) \} +  \sup_{z \in \RR^n} \{ f(z) -\Lambda^*_\mu(z) \}.$$
Thus,
\begin{equation} \label{laststep} \log \int e^f d\mu \leq \log \int e^{ \sup_{\lambda\in W} \{ \langle \lambda, x\rangle - \Lambda_{\mu}(\lambda)\}} d\mu(x) + \sup_{z \in \RR^n} \{ f(z) -\Lambda^*_\mu(z) \}\end{equation}
Note that $\nabla \Lambda_\mu(\RR^n) \subset K$ since any point in $\nabla \Lambda_\mu(\RR^n) $ is a barycenter of a probability measure supported on $K$  by \eqref{barycentermuh}. Therefore, for any $x\in K$, the function $\lambda \mapsto \langle \lambda, x\rangle - \Lambda_{\mu}(\lambda)$ is $D$-Lipschitz with respect to the $\ell^2$-norm. Let now $\delta>0$ and let $\mathcal{D}_{\delta}$ be a $\delta/D$-net of $V$. Then, 
$$ \log \int e^{ \sup_{\lambda\in W} \{ \langle \lambda, x\rangle - \Lambda_{\mu}(\lambda)\}} d\mu(x) \leq \delta + \log \int e^{ \sup_{\lambda\in \mathcal{D}_{\delta}} \{ \langle \lambda, x\rangle - \Lambda_{\mu}(\lambda)\}} d\mu(x).$$
Using a union bound, we get
$$ \log \int e^{ \sup_{\lambda\in \mathcal{D}_{\delta}} \{ \langle \lambda, x\rangle - \Lambda_{\mu}(\lambda)\}} d\mu(x) \leq 
\log |\mathcal{D}_{\delta}|,$$
which yields the claim.

\end{proof}

\begin{Rem}
In the case $\supp \mu \subset [0,1]^n$, one can get a potentially better bound by using not a net for the $\ell^2$-norm but for the $\ell^1$-norm instead. Indeed, for any $i\in\{1,\ldots,n\}$, and $\lambda \in \RR^n$, 
$$ \partial_i \Lambda(\lambda) = \int \langle x, e_i\rangle d \mu_y(x),$$
with $y= \nabla \Lambda(\lambda)$ and $\mu_y$ is defined in \eqref{defmuh}. Therefore, if $\supp \mu \subset [0,1]^n$, $\partial_i \Lambda(\lambda) \in [0,1]$. Thus, for fixed $x \in \supp \mu$, the function $\lambda \mapsto \langle \lambda, x\rangle -\Lambda_\mu(\lambda)$ is $1$-Lipschitz w.r.t the $\ell^1$-norm. In the last step \eqref{laststep} of the proof of Proposition \ref{partitionfunc}, if one takes $\delta>0$ and $\mathcal{F}_\delta$ to be a $\delta$-net for the $\ell^1$-norm, one obtains,
$$ \log \int e^{ \sup_{\lambda\in V} \{ \langle \lambda, x\rangle - \Lambda_{\mu}(\lambda)\}} d\mu(x) \leq \delta + \log |\mathcal{F}_\delta|.$$
This yields an error term in the approximation of the partition function getting closer to the one recently found in \cite[Proposition 5.1]{A} as a consequence of a structural result on probability measures on product spaces.
\end{Rem}

\begin{Rem} \label{gap1} As observed in Remark \ref{gappartition}, we see that the above proof of Proposition \ref{partitionfunc} yields an approximation of the partition function which seems a bit off from the one Lemma \ref{lbpartition} proposes.

In general, it is expected that when $f$ has a gradient of low complexity, this gap becomes negligible, that is, for any $y$ in an appropriate level set of $\Lambda^*_\mu$,
$$ \int f d\mu_y \simeq  f(y),$$
where $\mu_y$ is defined in \eqref{defmuh}.
 Indeed, if $\mu$ is the $n$-fold product measure $\nu^n$, with $\nu$ compactly supported for example, we note that by the mean-value theorem, we have,
$$ \big| \int f d\mu_y - f(y)\big| \leq \int \sup_{t \in K} |\langle \nabla f(t), x-y\rangle | d\mu_y(x),$$
where $K$ denotes the convex hull of the support of $\mu$. But, as $\mu_y$ has barycenter $y$, and is also a product measure with marginals having the same support as the one of $\nu$, we can deduce by the Majorization Theorem (see \cite[Theorem 12.16]{LT}),
$$ \big| \int f d\mu_y - f(y)\big|  \leq c\int \sup_{ t \in K} \langle \nabla f(t), x \rangle d\gamma^n(x),$$
where $c$ is a positive constant depending on the diameter of the support of $\nu$, and $\gamma^n$ is the standard Gaussian measure on $\RR^n$.
\end{Rem}

We now turn our attention to the proof of Theorem \ref{NL}. As the ``rate function'' we are aiming for, is not a priori convex, we cannot restrict ourselves to estimate the logarithmic Laplace transform of $f(X)$, and use the previous Theorem \ref{partitionfunc}. We will actually refrain ourselves from using Chebytchev's inequality, and prefer to work directly on the probability of deviations, in contrast with the path followed by Chatterjee and Dembo in \cite[section 4, proof of Theorem 1.1]{CD}.
\begin{proof}[Proof of Theorem \ref{NL}]To ease the notation, we write $\Lambda$, $\Lambda^*$ as shorthands for $\Lambda_\mu$ and $\Lambda^*_\mu$ respectively. The key element of this upper bound is a deterministic result which translates the event where $f(X) \notin V_\delta (\{ I\leq r\})$ into a large deviation event for the process, 
$$ \big( \langle \theta \lambda, X\rangle -\Lambda(\theta \lambda)\big)_{\theta >0, \lambda \in W}.$$
The bound of Theorem \ref{NL} is then obtained by using a net-argument and a union bound to control the deviations of the above process.

Fix $\kappa\geq  r>0$. Denote by $ K =\{ \Lambda^* \leq \kappa \}$.
In a first step, we translate in a more geometric language the event where $X \in K $ and $f(X) \notin V_\delta(\{I\leq r\})$, which is the object of the following lemma.
\begin{lemma}\label{lemmdet}Let $x \in K$. If $f(x) \notin V_\delta(\{ I \leq r\})$, then
$$ \inf_{ z \in \delta W^{\circ}} \Lambda^*(x-z) \geq r,$$
where $W^{\circ}$ is the polar set of $W$, that is, denoting by $h_W(z) = \sup_{\lambda \in W}\langle \lambda, z \rangle$ for $z \in \RR^n$ the support function of $W$,
$$W^{\circ} = \big\{z \in \RR^n, \ h_W(z)\leq  1\big\}.$$
\end{lemma}
Observe that since Theorem \ref{NL} is only relevant when $W$ is a small set in terms of complexity, its polar $\delta W^\circ$ can be thought as a large or typical set. The above lemma quantifies the fact that when $f(x) \notin V_\delta (\{ I\leq r\})$, $x$ is far away from $\delta W^\circ$ with respect to the cost function $\Lambda^*$.
\begin{proof}Let $x\in K$. By definition of $I$, for any $y\in \RR^n$ such that $\Lambda^*(y) \leq r$, we have $I(f(y)) \leq r$. Therefore, if $f(x) \notin V_{\delta}(\{I\leq r\})$, then
$$  \inf_{\Lambda^*(y) \leq r} |f(x) - f(y)| \geq \delta.$$
Since $\kappa \geq r$, we deduce from the assumption \eqref{meanvalueass} that  for any $y \in \RR^n$ such that $\Lambda^*(y) \leq r$,
$$ |f(x) - f(y)|\leq \sup_{\lambda \in W} \langle \lambda, x-y\rangle.$$
Therefore, we have 
\begin{equation} \label{ineqhw} \inf_{\Lambda^*(y) \leq r} \sup_{\lambda \in W} \langle \lambda, x-y \rangle\geq \delta.\end{equation}
Let $z \in \delta W^\circ$ and assume that $\Lambda^*(x-z) <r$. By continuity of $h_W$ and the continuity of  $\Lambda^*$ on its domain, we can assume $h_W(z)<\delta$. Using the above inequality \eqref{ineqhw} for $y =x -z$ we get that $h_W(x-y) = h_W(z) \geq \delta$, which yields a contradiction. Therefore,
$$ \inf_{ z \in \delta W^\circ} \Lambda^*(x-z) \geq r.$$
 \end{proof}
As one can observe, Lemma \ref{lemmdet} involves a convex optimization problem. A specific feature of such an optimization problem is that it admits a dual concave optimization problem, which is linked through its multipliers to the initial problem. Using this duality, we obtain the following lemma.

\begin{lemma}\label{boundoptimequiv}
 Let $x \in K$ and $\theta_0 = \kappa/\delta$. Then,
$$ \inf_{ z \in \delta W^{\circ}} \Lambda^*(x-z) \leq \sup_{ \underset{\lambda \in  W_{\theta}}{0 \leq \theta \leq \theta_0} } \{ \langle \theta \lambda , x\rangle - \Lambda(\theta\lambda)-\theta \delta\},$$
where for any $\theta>0$,
\begin{equation}\label{defWtheta}W_{\theta} = \big\{ \lambda \in W : \theta \lambda \in \mathrm{int}(\mathcal{D}_\Lambda), \  \nabla\Lambda(\theta \lambda)  \in K\big\}.\end{equation}
\end{lemma}
\begin{proof}
Since $\Lambda^*(x) \leq \kappa$ and $0\in \delta W^\circ$, we obtain by taking $z=0$, that 
\begin{equation} \label{finite} \inf_{z\in \delta W^{\circ}} \Lambda^*(x-z) \leq \kappa <+\infty.\end{equation}
As $\delta W^{\circ}$ is a closed set and $\Lambda^*$ has compact level sets by \cite[Lemma 2.2.31]{DZ}, the infimum is achieved at some $z_* \in \delta W^{\circ}$.

Since $\Lambda^*$ and $h_W$ are both convex functions, we deduce by Kuhn-Tucker Theorem (see \cite[Theorem 9.4]{Clarke}) that there exists $(\eta, \theta)\neq (0,0)$ with $\eta \in\{0,1\}$ and $\theta \geq 0$, such that $\theta(h_W(z_*) -\delta) = 0$, and
\begin{equation} \label{optimpb} \eta \Lambda^*(x-z_*) = \inf \big\{ \eta \Lambda^*(x-z) + \theta (h_W(z)-\delta) : z \in \RR^n \big\}.\end{equation}
 Evaluating the function on the right-hand side at $z=0$, we obtain that $\eta \Lambda^*(x-z_*) \leq -\theta \delta$. Therefore, the non-triviality condition $(\eta,\theta) \neq (0,0)$ implies that $\eta = 1$.
 By \eqref{finite}, we know that $x-z_* \in \mathcal{D}_\Lambda^*$. We claim that further:
\begin{fact}\label{interiorpoint}
 $x-z_*$ lies in the interior of $\mathcal{D}_{\Lambda^*}$. 
\end{fact}
 The intuition behind this fact is that the gradient of $\Lambda^*$ blows up around the frontier of its domain whereas the ``penalization function'' $h_W$ remains Lipschitz, so that it can only be beneficial for the minimum to lie in the interior of $\mathcal{D}_{\Lambda^*}$. We now make this idea rigorous and prove formally Fact \ref{interiorpoint}.
\begin{proof}
Note first that the barycenter of $\mu$, which we denote by $\overline{m}$, is in the interior of $\mathcal{D}_{\Lambda^*}$. Indeed, by Jensen's inequality, 
\begin{equation} \label{subdiff0} \forall \lambda \in \RR^n, \ \Lambda(\lambda) \geq \langle \lambda,\overline{m}\rangle.\end{equation}
Therefore, $\Lambda^*(\overline{m}) \leq 0$. Since $\Lambda^*\geq 0$, we deduce that $\Lambda^*(\overline{m})=0$ and in consequence $ 0 \in \partial \Lambda^*(\overline{m})$. As $\Lambda$ is strictly convex and essentially smooth, $\Lambda^*$ is also strictly convex and essentially smooth by \cite[Theorem 26.5]{Rockafeller}. In particular, $\Lambda^*$ is differentiable on $\mathrm{int}(\mathcal{D}_{\Lambda^*})$ and steep, so that any point with a nonempty subdifferential should lie in $\mathrm{int}(\mathcal{D}_{\Lambda^*})$. Therefore $\overline{m} \in \mathrm{int}(\mathcal{D}_{\Lambda^*})$.

To prove that $x-z_* \in \mathrm{int}(\mathcal{D}_\Lambda)$, we argue by contradiction. We will show that if $x-z_*$ lies on the frontier of $\mathcal{D}_{\Lambda^*}$, denoted by $\mathrm{fr}(\mathcal{D}_{\Lambda^*})$, then using the fact that $\Lambda^*$ is steep, we can find a barycenter of $\overline{m}$ and $x-z_*$ at which the cost function in the optimization problem \eqref{optimpb} is smaller than at $x-z_*$, thus violating the optimality of $x-z_*$. Assume now that $x-z_*\in \mathrm{fr}( \mathcal{D}_{\Lambda^*})$. Then, by \cite[Theorem 6.1]{Rockafeller}, we deduce that for any $t\in (0,1]$, $x-z_* - tu \in \mathrm{int}(\mathcal{D}_{\Lambda^*})$, where $ u = x-z_* -\overline{m}$, as we are taking a barycenter between an interior point and a frontier point of a convex set. Define
\begin{equation}\label{defF}\forall z \in \RR^n, \  F(z) = \Lambda^*(x-z) + \theta h_W(z).\end{equation}
For any $t>0$, there exists $\zeta_t \in  W$ such that $h_W(z_*) -h_W(z_*+tu) \geq -t \langle \zeta_t, u\rangle$. Using the convexity of $\Lambda^*$ we get, 
$$ F(z_*) \geq F(z_*+tu) + \langle  \nabla \Lambda^*(v_t) -\theta \zeta_t, t u\rangle,$$
where $v_t = x-z_*-tu$. Besides, $\Lambda^*$ is strictly convex and therefore achieves its unique minimum at $\overline{m}$ where $\Lambda^*(\overline{m})=0$. Thus for any $t<1$, $\Lambda^*(v_t)>0$. Since $v_t = \overline{m}+(1-t)u$, we deduce by convexity of $\Lambda^*$ that $\langle \nabla \Lambda^*(v_t) , u\rangle>0$ for any $t<1$. But $W$ is a bounded set, therefore $\zeta_t$ remains bounded as $t \in(0,1]$. Since $\Lambda^*$ is steep, we deduce that for $t$ close enough to $0$, 
$$ \langle \nabla \Lambda^*(v_t)  -\theta \zeta_t,  u\rangle>0,$$
which gives the desired contradiction.
\end{proof}

Since $z_*$ is a global minimum of the function $F$ defined in \eqref{defF}, we have $0 \in \partial F(z_*)$. 
From Fact \ref{interiorpoint}, we know that $x-z_*$ is an interior point of $\mathcal{D}_{\Lambda^*}$. As $\Lambda^*$ is differentiable on $\mathrm{int}(\mathcal{D}_{\Lambda^*})$, we have $0 \in \{- \nabla \Lambda^*(x-z_*)\} + \theta\partial h_W(z_*)$ by \cite[Theorem 4.10]{Clarke}.  By Danskin's formula (see \cite[Theorem 10.22]{Clarke}), the subdifferential of the support function $h_W$, is
$$\partial h_W(z) =\{ \lambda_* \in W: h_W(z) = \langle \lambda_*, z\rangle\}.$$
Thus, there is a $\lambda \in W$ satisfying $\delta = \langle\lambda , z_*\rangle$ such that,
\begin{equation} \label{crit} \nabla \Lambda^*(x-z_*) = \theta \lambda.\end{equation}
Therefore,
$$\Lambda^*(x-z_*) = \langle \theta \lambda, x-z_*\rangle - \Lambda (\theta\lambda) =  \langle \theta \lambda, x\rangle - \Lambda (\theta\lambda)-\theta \delta.$$
Note that $\theta$ and $\lambda$ depends on $x$ in an implicit way. However, we will show that $\theta$ and $\lambda$ can be restricted to certain subsets uniformly in $x$ such that $\Lambda^*(x) \leq \kappa$. More precisely, we will prove that $\theta \leq \kappa/\delta$ and $\lambda \in W_\theta$, with $W_\theta$ defined in \eqref{defWtheta}, thus ending the proof of  Lemma \ref{boundoptimequiv}.
We start with the bound on $\theta$. Using the convexity of $\Lambda^*$, we have,
$$\Lambda^*(x) -\Lambda^*(x-z_*) \geq \langle \nabla \Lambda^*(x-z_*), z_*\rangle.$$ Using that  $\langle \lambda, z_* \rangle = \delta$ and that $\Lambda^*$ is non-negative, we get $\Lambda^*(x) \geq \theta \delta$. Since $\Lambda^*(x)\leq \kappa$, we obtain that $\theta \leq \kappa/\delta$. To prove that $\lambda \in W_\theta$, we use the fact that $\nabla \Lambda$ and $\nabla \Lambda^*$ are inverse maps and \eqref{finite} to deduce that 
 $$\Lambda^*( \nabla \Lambda(\theta \lambda)) =  \Lambda^*(x-z_*)\leq \Lambda^*(x)\leq \kappa,$$
which finally ends the proof.

\end{proof}
Combining Lemmas \ref{lemmdet} and \ref{boundoptimequiv}, we obtain that if $x \in K$ and $f(x) \notin V_\delta(\{I\leq r\})$, then
\begin{equation} \label{implication} \sup_{ \underset{\lambda \in  W_{\theta}}{0 \leq \theta \leq \theta_0} } \{ \langle \theta \lambda , x\rangle - \Lambda(\theta\lambda)-\theta \delta\}\geq r,\end{equation}
where $W_\theta$ is defined in \eqref{defWtheta} for any $\theta>0$. In order to be allowed to take the probability of this event, we check the measurability of the supremum appearing in the above inequality.
\begin{fact}
The function $x \in \RR^n \mapsto  \sup_{ \underset{\lambda \in  W_{\theta}}{0 \leq \theta \leq \theta_0} } \{ \langle \theta \lambda , x\rangle - \Lambda(\theta\lambda)-\theta \delta\}$ is measurable.
\end{fact}
\begin{proof}
Since $(\theta,\lambda) \mapsto  \langle \theta \lambda , x\rangle - \Lambda(\theta\lambda)-\theta \delta$ is a continuous function, it suffices to prove that the subset 
$$ A : =\bigcup_{\theta \leq \theta_0} \{\theta\}\times W_\theta,$$
is separable in the sense that it contains a dense subset. Observe that $A$ is a bounded set since $W_\theta \subset W$ for any $\theta \geq0$ and $W$ is bounded by assumption. Therefore its closure, $\mathrm{cl}(A)$, is compact. This entails that $A$ is pre-compact, and in consequence that it is separable.

\end{proof}
Define now the measure, 
$$\PP_{K} = \PP( . \cap \{X \in K\}).$$
From \eqref{implication}, we have the upper bound,
\begin{equation} \label{unionprete} \PP_K\big( f(X) \notin V_{\delta}(\{I\leq r\}) \big)\leq \PP_K\big( \sup_{ \underset{\lambda \in  W_{\theta}}{0 \leq \theta \leq \theta_0}}  \{ \langle \theta \lambda , X\rangle - \Lambda(\theta\lambda) - \theta \delta\} \geq r\big).\end{equation}

In a last step, we control the deviations of the supremum of the process appearing in the right-hand side of \eqref{unionprete} using a net argument and performing a union bound. 
\begin{lemma}\label{netargument}  Let $D$ denote the diameter of  $K$. Let $\mathcal{D}_{\delta/2}$ be a $\delta/2D$-net for the $\ell^2$-norm of $W$. Then, 
$$\log  \PP_K\big( \sup_{ \underset{\lambda \in  W_{\theta}}{0 \leq \theta \leq \theta_0}}  \{ \langle \theta \lambda , X\rangle  - \Lambda(\theta\lambda) - \theta \delta\} \geq r\big) \leq -r +\log |\mathcal{D}_{\delta/2}| + \log\Big( \Big(\frac{\kappa LD}{\delta}\Big)\vee 1\Big)+1,$$
where $\theta_0 = \kappa/\delta$, $ L = \sup_{ \lambda \in W} ||\lambda ||_{\ell^2}$ and for any $\theta\geq 0$,
$$W_{\theta} = \big\{ \lambda \in W : \theta \lambda \in \mathrm{int}(\mathcal{D}_\Lambda), \  \nabla\Lambda(\theta \lambda)  \in K\big\}.$$
\end{lemma} 
\begin{proof}
 We start by a net argument on $\theta$. 
 For $x\in K$ fixed, define the function 
$$G : \theta \in \RR_+  \mapsto \sup_{\lambda \in W_{\theta}} \{ \langle \theta\lambda, x\rangle - \Lambda(\theta \lambda) - \theta \delta\}.$$
We claim that for any $\theta' \leq \theta$,
\begin{equation} \label{liponeside} G(\theta') - G(\theta) \leq (\theta -\theta')LD.\end{equation}
We observe first the following fact on the monotonicity of the sets $W_\theta$.
\begin{fact}\label{monotonicity}
The family $(W_{\theta})_{\theta \geq 0}$ is non-increasing for the inclusion.
\end{fact}
\begin{proof}
 Let $ \theta\geq 0$ and $\lambda \in W_\theta$. Since $K= \{ \Lambda^* \leq \kappa\}$, it means that $\theta \lambda\in\mathrm{int}(\mathcal{D}_{\Lambda^*})$ and $\Lambda^*(\nabla \Lambda(\theta \lambda ))\leq \kappa$. We know by \eqref{subdiff0} that $\overline{m} \in  \partial \Lambda(0)$. Since $\Lambda$ is essentially smooth and $\Lambda$ has a non-empty subdifferential at $0$, the point $0$ must lie in the interior of $\mathcal{D}_\Lambda$. Since by \cite[Theorem 6.2]{Rockafeller}, $\mathrm{int}(\mathcal{D}_\Lambda)$ is a convex set, we deduce that for any $0\leq \theta '\leq \theta $, $\theta' \lambda \in \mathrm{int}(\mathcal{D}_\Lambda)$.  In addition, for any $\theta' \leq \theta$ we can compute
\begin{align*}
\frac{\partial}{\partial \theta'}\Lambda^*(\nabla \Lambda(\theta' \lambda)) &= \langle \nabla \Lambda^*(\nabla \Lambda(\theta' \lambda)), \nabla^2\Lambda(\theta' \lambda). \lambda \rangle \\
& = \theta' \langle \lambda , \nabla^2\Lambda(\theta' \lambda). \lambda \rangle \geq 0,
\end{align*}
 where we used the fact that $\nabla \Lambda$ and  $\nabla \Lambda^*$ are inverse maps from one another. Thus, we can conclude that for any $\theta'\leq \theta$, $\Lambda^*(\theta' \lambda)\leq \Lambda^*(\theta \lambda)\leq \kappa $, and therefore $\lambda \in W_{\theta'}$.
\end{proof}
We now prove \eqref{liponeside}. Let $\theta' \leq \theta$. By convexity of $\Lambda$, we get for any $\lambda \in W_{\theta}$,
$$
\big( \langle \theta \lambda, X\rangle   -\Lambda(\theta' \lambda)\big) - \big ( \langle \theta' \lambda, X\rangle   -\Lambda(\theta' \lambda)\big)  \leq (\theta-\theta')\langle \lambda, \nabla \Lambda(\theta' \lambda) -X\rangle.$$
Since  $ W_{\theta}\subset W_{\theta'}$ by Fact \ref{monotonicity}, we have  $\nabla \Lambda (\theta'\lambda) \in K$. As we denoted by $D$ the diameter of $K$, we obtain
$$
\big( \langle \theta \lambda, X\rangle   -\Lambda(\theta' \lambda)\big) - \big ( \langle \theta' \lambda, X\rangle   -\Lambda(\theta' \lambda)\big)  \leq (\theta -\theta') L D,
$$
which holds for any $\theta'\leq \theta$ and $\lambda \in W_\theta\subset W_{\theta'}$. This yields the claim \eqref{liponeside}. 

Let $\mathcal{E}$ be a $1/(LD)$-net of the interval $[0,\theta_0]$.
 One can find a net $\mathcal{E}$ such that,
\begin{equation}\label{carditheta}|\mathcal{E}| \leq  \big(LD \theta_0\big)\vee 1 = \Big( \frac{\kappa LD }{\delta}\Big) \vee 1.\end{equation}
For any $\theta \in [0,\theta_0]$, one can find $\theta' \in \mathcal{E}$, such that $0 \leq \theta-\theta'\leq 1/LD$. From \eqref{liponeside}, we deduce that 
$$ \sup_{\theta \in [0,\theta_0]} G(\theta) \leq \sup_{\theta ' \in \mathcal{E}} G(\theta') + 1.$$
Thus, using a union bound, we get,
\begin{align}
 \PP_K& \big( \sup_{ \underset{\lambda \in  W_{\theta}}{0 \leq \theta \leq \theta_0}} \{ \langle \theta \lambda , X\rangle - \Lambda(\theta\lambda) - \theta \delta\} \geq r\big) \nonumber  \\
&\leq \sum_{\theta \in \mathcal{E}} \PP_K\big( \sup_{ \lambda \in  W_{\theta}}  \{ \langle \theta \lambda , X\rangle - \Lambda(\theta\lambda) - \theta \delta\} \geq r-1\big).\label{unionboundtheta}\end{align}
Let us now fix $\theta  \in \mathcal{E}$ and perform a net argument on $\lambda \in W_\theta$. Fix $x\in K$ and define the function $H$ taking values in $\RR \cup \{-\infty\}$,
$$ H : \lambda \in \RR^n \mapsto \langle \theta \lambda, X\rangle - \Lambda(\theta\lambda).$$
We  claim that for any $\lambda,\lambda'\in W_\theta$,
\begin{equation}\label{lipH} H(\lambda) - H(\lambda')\leq \theta D ||\lambda-\lambda'||.\end{equation}
Indeed, by convexity of $\Lambda$, we have for $\lambda,\lambda'\in W_\theta$,
$$ H(\lambda) - H(\lambda')\leq \theta \langle \lambda-\lambda',x-\nabla \Lambda(\theta \lambda')\rangle.$$
Since $\lambda'\in W_\theta$, we have $\nabla \Lambda(\theta \lambda') \in K$, which yields \eqref{lipH}.

Let $\mathcal{F}$ be a $\delta/D$-net for the $\ell^2$-norm of $W_\theta$ such that $\mathcal{F}\subset W_\theta$. Using \eqref{lipH} we obtain that
$$ \sup_{\lambda \in W_\theta} H(\lambda) \leq \sup_{\lambda'\in \mathcal{F}} H(\lambda') + \theta \delta.$$
Therefore, 
$$\PP_K\big( \sup_{\lambda \in W_{\theta}} \{ \langle\theta  \lambda, X\rangle - \Lambda(\theta\lambda) \} \geq r-1 +\theta \delta\big) \leq \PP_K\big( \sup_{\lambda \in \mathcal{F} } \{ \langle\theta \lambda, X\rangle - \Lambda(\theta\lambda) \} \geq r-1 \big).$$
Performing a union bound we get,
$$ \PP\big( \sup_{\lambda \in \mathcal{F}} \{ \langle \theta \lambda, X\rangle - \Lambda(\theta\lambda) \} \geq r-1 \big) \leq |\mathcal{F}| e^{-r+1},$$
where we used the fact that for any $\xi \in \RR^n$, and $t \geq 0$,
$$\PP(\langle \xi, X\rangle - \Lambda(\xi) \geq t )\leq e^{-t},$$
by Chernoff's inequality. Therefore,
\begin{equation}\label{nettheta}\PP_K\big( \sup_{\lambda \in W_{\theta}} \{ \langle\theta  \lambda, X\rangle - \Lambda(\theta\lambda) \} \geq r-1 +\theta \delta\big) \leq  |\mathcal{F}| e^{-r+1}.\end{equation}
 It remains now to show that we can find a $\delta/D$-net $\mathcal{F}$ of $W_\theta$ such that $\mathcal{F} \subset W_\theta$ and relate its cardinal to the one of a net of $W$.  For any $A\subset \RR^n$ and $r>0$ we denote by $N(A,rB_{\ell^2})$ and $\overline{N}(A,rB_{\ell^2})$ the following covering numbers
$$ N(A,rB_{\ell^2}) = \min \big\{ N \in\NN : \exists x_1,\ldots,x_N \in \RR^n, A \subset \bigcup_{i=1}^N B_{\ell^2}(x_i, r) \big\},$$
$$ \overline{N}(A,rB_{\ell^2}) = \min \big\{ N\in \NN : \exists x_1,\ldots,x_N \in A, A \subset \bigcup_{i=1}^N B_{\ell^2}(x_i, r) \big\}.$$
While it is true that the first notion of covering number $N(A,rB_{\ell^2})$ is non-decreasing in $A$ for the inclusion, this fact becomes wrong for $\overline{N}(A,rB_{\ell^2})$. When $A$ is convex, it is known (see \cite[Fact 4.1.4]{AGM}) that both covering numbers defined above coincide. Unfortunately, we cannot prove in general that $W_\theta$ is convex. But, up to loose a factor $2$ in the mesh of the net, we have the following fact:
\begin{fact}\label{factcovering} Let $A\subset B \subset \RR^n$ and $r>0$. Then, $  \overline{N}(A,rB_{\ell^2}) \leq N(B,(r/2)B_{\ell^2})$.
\end{fact}
\begin{proof}
For any $A\subset \RR^n$ and $r>0$, let  $M(A,rB_{\ell^2})$ be the separation number defined as 
$$ M(A,rB_{\ell^2}) = \max\big\{ N \in \NN : \exists x_1,\ldots,x_N \in A, \ \forall i\neq j, ||x_i-x_j||\geq r\big\}.$$
Let $A\subset B\subset \RR^n$. We know by \cite[Fact 4.1.11]{AGM} that $ \overline{N}(A,rB_{\ell^2}) \leq M(A,rB_{\ell^2})$. Since $A\subset B$,
$$ M(A,rB_{\ell^2}) \leq M(B,rB_{\ell^2}).$$
Using again \cite[Fact 4.1.11]{AGM}, we have
$$ M(B,rB_{\ell^2}) \leq N(B,(r/2)B_{\ell^2}),$$
which ends the proof.
\end{proof}
Let $\mathcal{D}_{\delta/2}$ be a $\delta/2D$-net for the $\ell^2$-norm of $W$. As $W_\theta \subset W$, we deduce from Fact \ref{factcovering} that 
 there exists a $\delta/D$-net $\mathcal{F}$ of $W_\theta$ for the $\ell^2$-norm such that $\mathcal{F} \subset W_\theta$ and $|\mathcal{F}|\leq |\mathcal{D}_{\delta/2}|$. From the two union bounds \eqref{unionboundtheta} and \eqref{nettheta}, we obtain
$$ \log \PP_K \big( \sup_{ \underset{\lambda \in  W_{\theta}}{0 \leq \theta \leq \theta_0}} \{ \langle \theta \lambda , X\rangle - \Lambda(\theta\lambda) - \theta \delta\} \geq r\big)\leq \log |\mathcal{E}| + \log |\mathcal{F}| -r+1.$$
Since $|\mathcal{E}|\leq (\kappa LD/\delta)\vee 1$ by \eqref{carditheta} and $|\mathcal{F} |\leq |\mathcal{D}_{\delta/2}|$, we deduce that
\begin{align*} \log \PP_K \big( \sup_{0 \leq \theta \leq \theta_0, \lambda \in  W_{\theta}} \{ \langle \theta \lambda , X\rangle &- \Lambda(\theta\lambda) - \theta \delta\} \geq r\big)\\
&\leq -r+\log |\mathcal{D}_{\delta/2}| + \log \Big(\Big(\frac{\kappa LD}{\delta}\Big)\vee 1\Big) + 1,\end{align*}
whih ends the proof.
\end{proof}
To finalize the proof, we choose $\kappa \geq r$ such that $\PP( X \notin K) \leq e^{-r}$. Using \eqref{unionprete}  and Lemma \ref{netargument}, we obtain
 \begin{align*}\PP\big( f(X) \notin V_{\delta}(\{I\leq r\}) \big)&\leq  \PP_K\big( f(X) \notin V_{\delta}(\{I\leq r\}) \big) +e^{-r}\\
& \leq \frac{\kappa LD}{\delta}|\mathcal{D}_{\delta/2}| e^{-r+1}+e^{-r}\\
& \leq \frac{2\kappa LD}{\delta}|\mathcal{D}_{\delta/2}| e^{-r+1},
\end{align*}
which gives us the claim.

\end{proof}

 \section{Large deviation of traces of Wigner matrices}
In this section, we will give a proof of Theorem \ref{LDPtrb} and  Corollary \ref{sharpsubGtr}.

\subsection{Large deviation upper bound}

We start with the large deviation upper bound of Theorem \ref{LDPtrb}. The  strategy for this upper bound will be similar to the one we adopted to investigate the large deviations of the moments of $\beta$-ensembles in \cite[section 3]{LDPtr}. It consists in truncating the spectrum so as to reduce ourselves to study the deviations of a small fraction of the eigenvalues. Once this truncation made, we will be able to use efficiently Theorem \ref{NL}.

We introduce some notation we will use throughout this section. Let $Y\in \mathcal{H}_n$.  For any $k \in \{1,\ldots ,n\}$, we denote by $\tr_{[k]} Y$ the truncated trace:
\begin{equation} \label{trunctrdef} \tr_{[k]} Y = \sum_{i=1}^k \lambda_i(Y),\end{equation}
where $\lambda_1(Y),\ldots,\lambda_n(Y)$ are the eigenvalues of $Y$ in non-increasing order. For any measurable function $f : \RR \to \RR$, we define $f(Y)$ by functional calculus as the Hermitan matrix,
$$ f(Y) = \sum_{i=1}^n f(\lambda_i(Y)) u_iu_i^*,$$
where $u_1,\ldots,u_n$ are the eigenvectors of $Y$ associated to $\lambda_1(Y),\ldots, \lambda_n(Y)$.
We now define, for $d$ even,
\begin{equation}\label{trunctr} f_k(Y) =  \frac{1}{n}\tr_{[k]} (Y/\sqrt{n})^d,\end{equation}
whereas for $d$ odd,
$$ f_k(Y) = \frac{1}{n}\tr_{[k]} (Y_+/\sqrt{n})^d - \frac{1}{n}\tr_{[k]} (Y_-/\sqrt{n})^d.$$
The first step toward the proof of the upper bound is the following lemma, which will rely on concentration arguments.
\begin{Lem}\label{equivexpo}Assume $X$ is a Wigner matrix satisfying the convex concentration property.
Let $k\in \{1,\ldots,n\}$ such that $n^{\frac{1}{d-1}} = o(k)$ and $k=o(n)$. For any $t>0$,
$$\lim_{n\to +\infty} \frac{1}{n^{1+\frac{2}{d}} } \log \PP \Big( \big|\frac{1}{n}\tr (X/\sqrt{n})^d - f_{k}(X)- \mu_{sc}(x^d) \big| >t \Big) = - \infty.$$

\end{Lem}

\subsubsection{Concentration inequalities}\label{conc}
In order to prove Lemma \ref{equivexpo}, we will develop some deviations inequalities for the number of eigenvalues falling outside an interval and for truncated linear statistics of Hermitian random matrices satisfying the convex concentration property defined in \ref{concconv}. The proof of these inequalities will follow the now classical path laid by the work of Guionnet and Zeitouni in \cite{GZconc}.

\begin{Pro}\label{conckconv}  Let $X$ be a random Hermitian matrix satisfying the convex concentration property for some constant $c>0$. Let $k \in \{1,\ldots,n\}$ and let $f : \RR \to \RR$ be a convex $1$-Lipschitz function which achieves its infimum on $\RR$. For any $t>0$,
$$\PP \Big( \big| \tr_{[k]} f(X) - \EE \tr_{[k]} f(X)\big| > t \Big) \leq c^{-1}\exp\Big( -\frac{ ct^{2}}{k}\Big),$$
where $\tr_{[k]}$ is the truncated trace defined in \eqref{trunctrdef}.
\end{Pro}
In order to apply our convex concentration property, we need to prove that the truncated linear statistics are convex Lipschitz functions of the entries. This is the object of the following lemma.

\begin{Lem}\label{convex}
Let $f : \RR \to \RR$ be a convex function which achieves its infimum on $\RR$. Let $k\in \{1,\ldots,n\}$. The function $T_f$ defined by,
$$\forall X \in \mathcal{H}_n, \ T_f(X) = \tr_{[k]} f(X),$$
 is  convex. Moreover, if $f$ is $1$-Lipschitz, then $T_f$ is $\sqrt{k}$-Lipschitz with respect to the Hilbert-Schmidt norm.
\end{Lem}

\begin{proof}\renewcommand{\qedsymbol}{} The proof is a variation around the one of Klein's lemma (see  \cite[Lemma 4.4.12]{AGZ}). Since $f$ achieves its infimum, we know that there exists $x_0$ such that $0 \in \partial f(x_0)$. Considering $\tilde{f} = f(.+x_0)-f(x_0)$, we may and will assume that $x_0 = 0$ and $f(0)=0$.
We will show the following representation of $T_f$ as supremum of affine functions.
\end{proof}
\begin{Lem}\label{fnsup} Let $f:\RR \to \RR$ be a convex function.
Assume $f(0)=0$ and $0\in \partial f (0)$. Let $\zeta : \RR \to \RR$ be a function such that, 
$$ \forall x \in \RR, \ \zeta(x) \in \partial f(x), \text{ and } \zeta(0)=0.$$
Then, for any $X \in \mathcal{H}_n$,
\begin{equation} \label{respresaff} T_f(X) = \sup_{ \rk Y \leq k} \{ \tr f(Y) + \tr \zeta(Y)(X-Y) \}.\end{equation}
Moreover, if $f$ is $1$-Lipschitz, then $T_f$ is $\sqrt{k}$-Lipschitz with respect to the Hilbert-Schmidt norm.
\end{Lem}
\begin{proof}
 First, we know by \cite[Theorem 23.4]{Rockafeller} that for every $x \in \RR$, $\partial f (x) \neq \emptyset$, which justifies the existence of $\zeta$.
Let $\lambda_1,\ldots,\lambda_n$ be the eigenvalues of $X$ such that $f(\lambda_1)\geq \ldots \geq  f(\lambda_n)$. Let $u_1,\ldots,u_n$ be the associated eigenvectors. Note that because we assumed $f(0)=0$, then if we take $Y = \sum_{i=1}^k \lambda_i u_i u_i^*$, we get the equality,
$$ T_f(X) = \tr f(Y) + \tr \zeta(Y)(X-Y),$$
 using the orthogonality of the eigenvectors. Therefore only the inequality is left to prove. Let $Y \in\mathcal{H}_n$ with rank less that $k$, and denote by $e_1,\ldots,e_n$ an orthonormal basis of eigenvectors such that $e_{k+1},\ldots,e_n$ are in the kernel, and by $\mu_1,\ldots,\mu_k$ the eigenvalues associated to $e_1,\ldots,e_k$. 
From the convexity of $f$, we have for any $j \in \{1,\ldots,n\}$, $i \in \{1,\ldots,k\}$, 
$$f(\lambda_j) \geq f(\mu_i) + \zeta(\mu_i)(\lambda_j-\mu_i).$$
Multiplying the above inequality by $|\langle u_j, e_i \rangle|^2$, and summing over $j \in\{1,\ldots,n\}$, we get
\begin{equation}\label{convtruntr} \sum_{j=1}^n | \langle u_j, e_i \rangle|^2 f(\lambda_j) \geq f(\mu_i) + \sum_{j=1}^n |\langle u_j, e_i \rangle|^2 \zeta(\mu_i) (\lambda_j -\mu_i).\end{equation}
Writing the trace of $f'(Y)(X-Y)$ in the basis of the $e_i$'s, we have
\begin{align*}
 \tr \zeta(Y)(X-Y)& = \sum_{i=1}^n \langle e_i,  \zeta(Y)(X-Y) e_i\rangle \\
&= \sum_{i=1}^n \langle \zeta(Y)e_i,  (X-Y) e_i\rangle\\
& = \sum_{i=1}^k \langle \zeta(Y)e_i,  (X-Y) e_i\rangle,
\end{align*}
where we used the fact that $\zeta(0)=0$. Finally, using the spectral decomposition of $X$, we get
$$  \tr \zeta(Y)(X-Y) = \sum_{i=1}^k \sum_{j=1}^n |\langle u_j, e_i \rangle|^2 \zeta(\mu_i) (\lambda_j -\mu_i).$$
Summing \eqref{convtruntr} now over $i\in\{1,\ldots,k\}$, we deduce, writing $\alpha_j = \sum_{i=1}^k |\langle u_j, e_i\rangle|^2$, that
$$\sum_{j=1}^n \alpha_j f(\lambda_j) \geq \tr f(Y) +  \tr \zeta(Y)(X-Y) .$$
Observe that $\alpha_j \in [0,1]$ and $\sum_j \alpha_j =k$. The maximum 
$$ \max \{ \sum_j \alpha_j f(\lambda_j) : \alpha_j \in [0,1],  \ \sum_j \alpha_j =k\},$$
is achieved at the vector $\alpha = \Car_{J}$, where $J$ is the set of indices $j$ corresponding to the $k$ highest values of $f(\lambda_j)$. 
This shows that the representation \eqref{respresaff} holds, and that $T_f$ is a convex function.

Assume further that $f$ is $1$-Lipschitz. This entails that $|\zeta(x)|\leq 1$ for any $x\in\RR$. Consequently, $||\zeta(Y)||_2 \leq \sqrt{k}$ for any $Y\in\mathcal{H}_n$ with $\rk(Y) \leq k$. We conclude from the representation \eqref{respresaff} that $T_f$ is $\sqrt{k}$-Lipschitz w.r.t the Hilbert-Schmidt norm.
\end{proof}

As a consequence of Proposition \ref{conckconv}, we get the following deviation estimate on the number of eigenvalues present outside the bulk.
\begin{Pro}\label{probatrou} Let $X$ be a random Hermitian matrix satisfying the convex concentration property with constant $c>0$. Denote by $\lambda_1(X)$ its top eigenvalue and $||X||$ its operator norm.
For any $M\geq 4\EE( \lambda_1(X))_+$, and $1\leq k \leq n$,
$$ \PP \big(  \mathcal{N}[M,+\infty) \geq k \big) \leq   c^{-1}\exp \Big( -\frac{c M^2 k}{16}\Big),$$
where for any $I \subset \RR$, $\mathcal{N}(I)$ denotes the number of eigenvalues of $X$ in $I$.
As a consequence, for any $M\geq 4\EE ||X||$  and $1\leq k \leq n$,
$$\PP \big(  \mathcal{N}\big((-M,M)^c\big) \geq k \big) \leq 2c^{-1}\exp\left( -\frac{c  M^2k}{32}\right).$$
\end{Pro}
\begin{proof}
 Let $f(x)= \frac{2}{M}(x-M/2)_+$ for any $x\in \RR$.
 Since $f(x) \geq 1$ for $x\geq M$,  we have, denoting by $\lambda_1(X),\ldots,\lambda_n(X)$ the eigenvalues of $X$ in non-increasing order,
$$\PP \big(  \mathcal{N}[M,+\infty) \geq k \big) \leq  \PP\big( \sum_{i=1}^k f(\lambda_i(X)) \geq k\big).$$
But 
$$\EE\Big( \sum_{i=1}^k f(\lambda_i(X)) \Big)\leq \frac{2}{M}k \EE (\lambda_1(X))_+ \leq \frac{k}{2},$$
for $M\geq 4\EE (\lambda_1(X))_+$. Thus,
$$\PP \big(  \mathcal{N}[M,+\infty) \geq k \big) \leq  \PP\Big( \sum_{i=1}^k f(\lambda_i(X))- \EE \sum_{i=1}^k f(\lambda_i(X)) \geq k/2\Big).$$
Applying Proposition \ref{conckconv}, we deduce that
\begin{equation} \label{dev} \PP \big(  \mathcal{N}[M,+\infty) \geq k \big) \leq   c^{-1}e^{-\frac{ck M^2}{16}},\end{equation}
which gives the first claim.

Since by definition $\lceil k/2 \rceil -1 <k/2$, we have using a union bound,
$$\PP \big(  \mathcal{N}\big((-M,M)^c\big) \geq k \big) \leq \PP \big(  \mathcal{N}(-\infty,-M] \geq \ell \big)  +\PP \big(  \mathcal{N}[M,+\infty) \geq \ell \big),$$
where $\ell = \lceil k/2 \rceil$.
For any $M\geq 4 \EE || X||$, we obtain by applying inequality \eqref{dev} to $X$ and $-X$ alternatively, 
$$\PP \big(  \mathcal{N}\big((-M,M)^c\big) \geq k \big) \leq  2c^{-1}e^{-\frac{c k M^2}{32}},$$
which gives the second claim.

\end{proof}

\subsubsection{An exponential equivalent}
In this section, we apply the concentration inequalities we obtained in the previous section to give a proof of the exponential equivalent of Lemma \ref{equivexpo}.
\begin{proof}[Proof of Lemma \ref{equivexpo}]Let $k$ such that  $n^{\frac{1}{d-1}} = o(k)$ and $k=o(n)$. We will prove the claim only in the case where $d$ is even, the case where $d$ is odd being almost the same.
Let $M>0$, and define the truncated power function, by $f_M(x) = x^d$ for $|x|\leq M$, and for $|x|>M$ by,
$$f_M(x) = 
dM^{d-1}(|x|-M) +M^d,$$
Let $\lambda^*_1,\ldots,\lambda^*_n$ be the eigenvalues of $X/\sqrt{n}$ in decreasing absolute values. We can write,
$$ \sum_{i=k+1}^n f_M(\lambda_i^*) = \tr f_M(X/\sqrt{n}) - \tr_{[k]} f_M(X/\sqrt{n}).$$
As $f_M(./\sqrt{n})$ is $dM^{d-1}/\sqrt{n}$-Lipschitz, applying Proposition \ref{conckconv} alternatively to $\tr f_M(X/\sqrt{n})$ and $\tr_{[k]} f_M(X/\sqrt{n})$, and performing a union bound, we deduce that for any $t>0$,
\begin{equation} \label{conctrunc} \PP\Big( \big|\sum_{i=k+1}^{n} f_M(\lambda_i^*)  - \EE \sum_{i=k+1}^{n} f_M(\lambda_i^*)  \big| >tn \Big) \leq 2c^{-1}\exp\Big( - \frac{c n^2 t^2}{4d^2 M^{2(d-1)}}\Big),\end{equation}
where we used the fact that $k\leq n$. On one hand,
\begin{align*}
\big|\EE \sum_{i=k+1}^{n} {\lambda_i^*}^d -  \EE \sum_{i=k+1}^{n} f_M(\lambda_i^*)\big|& \leq  \EE \sum_{i=1}^n |\lambda_i^*|^d \Car_{|\lambda_i^*|\geq M} \\
&\leq \frac{1}{M}\EE \tr |X/\sqrt{n}|^{d+1}.
\end{align*} 
 Using twice Cauchy-Schwarz inequality, we get
$$  \frac{1}{n} \EE \tr |X/\sqrt{n}|^{d+1} \leq \EE \big(\frac{1}{n}\tr (X/\sqrt{n})^{2(d+1)}\big)^{\frac{1}{2}} \leq  \big(\EE\frac{1}{n}\tr (X/\sqrt{n})^{2(d+1)}\big)^{\frac{1}{2}}=O(1),$$
where we used in the last equality \cite[Lemma 2.1.6]{AGZ}. Therefore,
$$ \big|\EE \sum_{i=k+1}^{n} {\lambda_i^*}^d -  \EE \sum_{i=k+1}^{n} f_M(\lambda_i^*)\big|=O\Big(\frac{n}{M}\Big).$$
On the other hand,
$$ |\EE\sum_{i=k+1}^{n} {\lambda_i^*}^d - \EE \tr (X/\sqrt{n})^d |\leq k\EE || X/\sqrt{n}||^d = O(k),$$
as $\EE || X/\sqrt{n}||^d =O(1)$ by \cite[Exercice 2.1.27]{AGZ}. 
Therefore, if $k=o(n)$ and $M$ goes to infinity with $n$, then as  $\EE \frac{1}{n} \tr (X/\sqrt{n})^d$ converges to $\mu_{sc}(x^d)$, we have,
\begin{equation} \label{convertr}\big|\mu_{sc}(x^d) -  \EE\Big( \frac{1}{n} \sum_{i=k+1}^{n} f_M(\lambda_i^*)\Big)\big|
\underset{n\to +\infty}{\longrightarrow} 0.\end{equation}

Fix $t>0$. Now let $M$ go to infinity with $n$ such that $n^{1+2/d} = o(n^2/M^{2(d-1)})$, that is $M^2 =o(n^{a})$ with $a = (d-2)/(d(d-1))$. The above inequality \eqref{conctrunc} and the convergence \eqref{convertr} give 
\begin{equation} \label{expoeqM} \lim_{n\to +\infty} \frac{1}{n^{1+\frac{2}{d}}} \log  \PP\Big(\big | \frac{1}{n}\sum_{i=k+1}^{n} f_M(\lambda_i^*)  - \mu_{sc}(x^d) \big| >t \Big)  = -\infty.\end{equation}
But since  $\EE || X || =O(\sqrt{n})$ by \cite[Exercice 2.1.29]{AGZ}, we deduce from Proposition \ref{probatrou} that in order to have that 
\begin{equation} \label{gap} \lim_{n\to +\infty} \frac{1}{n^{1+\frac{2}{d}}} \log  \PP\Big(\mathcal{N}\big([-R\sqrt{n},R\sqrt{n}]^c\big)\geq k \Big)  = -\infty,\end{equation}
it is sufficient to take $R$ going to $+\infty$ with $n$ such that $n^{2/d} = o(kR^2)$. We assumed that $n^{\frac{1}{d-1}} = o(k)$, which is equivalent to say that,
$$n^{\frac{2}{d}} = o(kn^a).$$ 
Therefore, we can find $M$ going to infinity which satisfies both conditions,
$$ M^2 = o(n^a), \text{ and } n^{\frac{2}{d}} = o(kM^2).$$
With this choice of $M$, both estimates \eqref{expoeqM}, and \eqref{gap} with $R=M$, hold.
 But then,
$$\PP \Big( \big| \sum_{i=k+1}^n {\lambda_i^*}^d - \sum_{i=k+1}^n f_M(\lambda_i^*)  \big| > t n \Big) \leq \PP\Big(\mathcal{N}([-M\sqrt{n},M\sqrt{n}]^c)\geq k \Big) $$
Therefore, 
$$\limsup_{n\to +\infty} \frac{1}{n^{1+\frac{2}{d}}} \log \PP \Big( \big| \sum_{i=k+1}^n {\lambda_i^*}^d - \sum_{i=k+1}^n f_M(\lambda_i^*)  \big| > t n \Big)=-\infty,$$
which together with \eqref{expoeqM} give the claim.

\end{proof}

\subsubsection{Proof of the large deviation upper bound}\label{lbtr}

From Lemma \ref{equivexpo}, we see that it suffices to understand the large deviations of a truncated trace $f_{k}(X)$ defined in \eqref{trunctr}, with $n^{\frac{1}{d-1}} = o(k)$, and $k=o(n)$. The point is that this truncation lowers significantly the complexity. Indeed, we will be able to encode by $O(nk\log n)$ bits the ``gradient'' of the truncated trace $\tr_{[k]}(X^d)$. Thus, Proposition \ref{NL} will give us a relevant upper bound, with respect to the speed $n^{1+\frac{2}{d}}$, as soon as we can take $k\log n = o(n^{\frac{2}{d}})$. But, for $d\geq 3$, we see that
$$ \frac{1}{d-1} < \frac{2}{d}.$$
Therefore, we can and will take $k$ which satisfies both conditions $n^{\frac{1}{d-1}} = o(k)$ and $k\log n =o(n^{\frac{2}{d}})$.

\textbf{Property of the rate function $I_+$.}
We begin with proving that the rate function $I_+$ defined in Theorem \ref{LDPtrb} is a good rate function. It will follow from the next lemma.

\begin{Lem}\label{lsci}
Let $c_n : \RR^n \to \RR_+$ and  $F_n : \RR^n \to \RR$ be families of functions indexed by $n\in\NN$, such that for any $r>0$, the subset
$$ \bigcup_{n\in \NN}F_n(\{ c_n \leq r \}),$$
is bounded. Then the functions defined as,
$$ \forall x \in \RR, \ I^{(n)}(x) = \sup_{\delta>0} \inf_{m\geq n} I_{m,\delta}(x),$$
and 
$$ \forall x \in \RR, \ I_+(x) = \sup_{\delta>0} \liminf_{n\to +\infty} I_{n,\delta}(x),$$
with
$$ I_{m,\delta}(x) = \inf \big\{ c_m(h) : | F_m(h) -x|<\delta, h \in\RR^m \big\},$$
are good rate functions.
\end{Lem}
\begin{proof}
Let $\tau>0$. We can write, for any $x \in\RR$,
\begin{align*}
I^{(n)}(x) \leq \tau & \Longleftrightarrow \forall \delta >0, \inf_{m\geq n } I_{m,\delta}(x) \leq \tau\\
& \Longleftrightarrow  \forall \delta >0, \forall \eps>0, \exists m\geq n,  I_{m,\delta}(x) <\tau +\eps.
\end{align*}
By definition of $I_{m,\delta}$, we get\begin{align*}
I^{(n)}(x) \leq \tau & \Longleftrightarrow   \forall \eps>0, \forall \delta >0,   \exists m\geq n,  \exists h\in \RR^m, c_m(h) <\tau +\eps, |F_m(h) - x|<\delta.
\end{align*}
Thus,
$$ \{ I^{(n)} \leq \tau \} = \bigcap_{\eps>0} \overline{ \bigcup_{m\geq n } F_m\big(\{ c_m \leq \tau+\eps\}\big )}.$$ 
This yields that $I^{(n)}$ is a good rate function. As $I_+ = \sup_{n\in\NN} I^{(n)}$, we deduce that $I_+$ is also a good rate function.
\end{proof}
Let us now check that the assumptions of Lemma \ref{lsci} are fulfilled in our setting. We first  show that $X$ is sub-Gaussian in the sense that there exists $C>0$ such that 
 \begin{equation} \label{subG} \forall Y \in \mathcal{H}_n, \ \Lambda^*(Y) \geq \frac{1}{2C^2} \tr (Y^2).\end{equation}
Applying the convex concentration property defined in \ref{concconv},  to the function $f(X) =\tr (XH)$ with $H \in\mathcal{H}_n$ fixed, we obtain that for any $t>0$,
$$ \PP\big( | \tr (XH)|>t \big) \leq c^{-1}e^{- \frac{c t^2}{\tr (H^2)}}.$$
We deduce that $\Lambda(H) \leq \frac{C^2}{2} \tr (H^2)$, which gives \eqref{subG}.
But, using the fact that $x\in \RR_+ \mapsto x^{2/d}$ is sub-additive, we have for any $Y\in\mathcal{H}_n$, $|\tr( Y^d)| \leq \tr( |Y|^d)\leq ( \tr (Y^2))^{d/2}$, so that whenever 
$$\Lambda^*(Y) \leq r n^{1+\frac{2}{d}},$$
for some $r>0$, then we have $| \frac{1}{n} \tr(Y/\sqrt{n})^d | \leq (2C^2r)^{\frac{d}{2}}$. This shows that we can apply Lemma \ref{lsci} with $c_n(Y) = n^{-(1+\frac{2}{d})}\Lambda^*(Y)$ and $F_n(Y) = \mu_{sc}(x^d) + \frac{1}{n}\tr(Y/\sqrt{n})^d$, and conclude that $I_+$ is a good rate function.

\textbf{Upper bound. }We can now proceed with the proof of the  upper bound. 
By \cite[Theorem 4.2.13]{DZ}, it is sufficient to prove the large deviations upper bound for the sequence $(f_k(X))_{n\in \NN}$, as $(\mu_{sc}(x^d)+ f_k(X))_{n\in \NN}$ is exponentially equivalent to $( \frac{1}{n} \tr(X/\sqrt{n})^d)_{n\in \NN}$ by Lemma \ref{equivexpo}.  We will first make sure that the rate function we are going to obtain by applying Theorem \ref{NL} is the same as the one we are aiming for, that is:

\begin{Lem}\label{ratefunctruc}

Assume that $k$ goes to  $+\infty$ with $n$. For any $n\in \NN$, $\delta >0$, set
$$ \forall  x \in \RR, \ J_{n,\delta}(x) = \inf \big\{ \frac{1}{n^{1+\frac{2}{d}}} \Lambda^*(Y) : |f_k(Y)-x|<\delta , Y \in \mathcal{H}_n\big\},$$
and let
$$ J_+ =  \sup_{\delta>0} \liminf_{n\to +\infty} J_{n,\delta},$$
where $f_k$ is defined in \eqref{trunctr}. 
For any $x \in \RR$,
$$ I_+(x+\mu_{sc}(x^d)) = \sup_{\delta>0}  \liminf_{n\to +\infty} J_{n,\delta}(x),$$
where $I_+$ is defined as
$$\forall x \in \RR, \  I_+(x) =\sup_{\delta>0}\liminf_{n\to +\infty} I_{n,\delta}(x),$$
and for any $n\in \NN$, $\delta>0$,
$$ I_{n,\delta}(x) = \inf \Big\{ \frac{\Lambda^*(Y)}{n^{1+\frac{2}{d}}} : |\frac{1}{n} \tr (Y/\sqrt{n})^d - \mu_{sc}(x^d) - x| <\delta, Y \in \mathcal{H}_n \Big\}.$$ 
\end{Lem}

\begin{proof}
We will prove that $I_+(.+\mu_{sc}(x^d))$ and $J_+$ have the same level sets. We will first observe that if $Y\in\mathcal{H}_m$ is a non-negative matrix such that $\tr (Y^2) = O(m^{1+\frac{2}{d}})$ then,
\begin{equation} \label{obs} \tr (Y^d) - \tr_{[k]} (Y^d) = o( m^{1+\frac{d}{2}}).\end{equation}
Let $\lambda_1\geq\ldots\geq \lambda_m\geq 0$ be the eigenvalues of $Y$. From the fact that $\tr (Y^2) = O(m^{1+ \frac{2}{d}})$, we deduce that for any $1 \leq \ell \leq m$,
$$ \lambda_{\ell} \leq \ell^{-\frac{1}{2}}O(m^{\frac{1}{2}+\frac{1}{d}}).$$ 
Thus,
$$ \tr (Y^d) - \tr_{[k]}( Y^d)  = \sum_{\ell>k} \lambda_{\ell}^d = O(m^{1+\frac{d}{2}} k^{1-\frac{d}{2}}),$$
which gives \eqref{obs}.
Let now $\tau>0$ and $x$ such that $J_+(x) \leq \tau$. Let $\delta>0$. For any $n\in \NN$, we have,
$$ \inf_{m\geq n} J_{m,\delta}(x)\leq \tau.$$
Therefore,
$$ \inf_{m\geq n} J_{m,\delta}(x) = \inf_{m\geq n} \inf_{Y\in\mathcal{H}_m} \big \{ \frac{\Lambda^*(Y)}{m^{1+\frac{2}{d}}}  :  |f_k(Y) -x|<\delta, \Lambda^*(Y) \leq 2\tau m^{1+\frac{2}{d}}\big\}$$ 
For any $m\in\NN$ and $Y\in \mathcal{H}_n$ such that $\Lambda^*(Y) \leq 2\tau m^{1+\frac{2}{d}}$, we know by \eqref{subG} that 
$$ \tr(Y^2) \leq 4C^2 \tau m^{1+\frac{2}{d}}.$$
 Applying the observation \eqref{obs} to $Y_+$ and $Y_-$ alternatively, we deduce that $\tr(Y^d) - f_k(Y)=o(m^{1+\frac{d}{2}})$. Therefore, for $n$ large enough, 
$$ \inf_{m\geq n} J_{m,\delta}(x) \geq  \inf_{m\geq n} I_{m,2\delta}(x).$$
Therefore,
$$ \liminf_{n\to +\infty} I_{n,2\delta}(x) \leq \tau.$$
As the above inequality is true for any $\delta>0$, we obtain $I_+(x)\leq \tau$. Inverting the roles of $I_+$ and $J_+$, we get the other inclusion.

\end{proof}

We come back to the proof of the upper bound of Theorem \ref{LDPtrb}. Let $F$ be a closed subset of $\RR$. We want to prove that,
$$ \limsup_{n\to +\infty} \frac{1}{n^{1+\frac{2}{d}}} \log \PP( f_k(X) \in F) \leq - \inf_F J_+,$$
where $J_+$ is defined in Lemma \ref{ratefunctruc}.
We can assume without loss of generality that $\inf_F J_+>0$. Let $r>0$ such that $\inf_F J_+>r$. Put in another way,
$$F\cap \{ J_+\leq r\} = \emptyset.$$
As $J_+$ is a good rate function, we can find a $\delta>0$ such that 
$$ F \cap V_{2\delta}(\{J_+ \leq r\}) = \emptyset.$$ Define for any $n\in \NN$,
$$ J^{(n)} = \sup_{\delta>0} \inf_{m\geq n } J_{m,\delta}.$$ 
Note that as $J_+ = \sup_{n\in \NN} J^{(n)}$,
$$\{ J_+ \leq r\} = \bigcap_{n\in \NN} \{ J^{(n)} \leq r \}.$$
Using \eqref{subG} and the fact that for any $Y\in \mathcal{H}_n$,
$$ f_k(Y) \leq (\tr (Y^2))^{d/2},$$
we deduce by Lemma \ref{lsci} that $J^{(n)}$ is a good rate function.
 Therefore, $\{ J^{(n)} \leq r\}$ is a non-increasing sequence of compact subsets, so that for $n$ large enough,  
$$\{ J^{(n)} \leq r \} \subset V_{\delta}( \{ J_+ \leq r \}).$$
Therefore, 
$$ F \cap V_{\delta}\big( \{ J^{(n)} \leq r \}\big) =\emptyset.$$
But, $\{ J^{(n)} \leq r \} \supset \{ J \leq rn^{1+\frac{2}{d}} \}$, where
$$ J(x) = \inf \big\{ \Lambda^*(Y) :f_k(Y) = x , Y \in\mathcal{H}_n\big\}.$$ 
Thus,
\begin{equation} \label{reductionup} \PP\big( f_k(X) \in F\big) \leq \PP\big( f_k(X) \notin V_{\delta}( \{ J \leq rn^{1+\frac{2}{d}} \}) \big).\end{equation}
We are now in the position of applying Theorem \ref{NL}. First, we check that the tightness condition \eqref{tightness} is satisfied. By  \cite[Lemma 5.1.14]{DZ}, we know that for any $i,j$,
$$\EE e^{\frac{1}{2} \Lambda_{i,j}^* (X_{i,j}) } \leq 4^{(1+\Car_ {i \neq j})},$$
using the fact that $(\Re X_{i,j}, \Im X_{i,j})$ are independent. Therefore, by Chernoff's inequality,
$$ \PP( \Lambda^*(X)> 8n^2)\leq  e^{-4n^2} 4^{n^2} \leq e^{-n^2},$$
which proves that \eqref{tightness} is fulfilled.

 Let $K = \{\Lambda^* \leq 8 n^2\}$. 
We will now bound the increments of $f_k$ on $K$. As noted before, the level sets of $\Lambda^*$ are included in Hilbert-Schmidt balls, more precisely, we know from \eqref{subG},
\begin{equation} \label{inclu}K \subset 4 C nB_2,\end{equation}
where $B_{2}$ denotes the ball of radius $1$ for the Hilbert-Schmidt norm.
By Lemma \ref{fnsup}, we know that if $f : \RR \to \RR$ is a convex differentiable function such that $f(0)=f'(0)=0$, then the function
\begin{equation*} \forall Y \in \mathcal{H}_n, \ T_f(Y) = \tr_{[k]} f(Y),\end{equation*}
admits the variational representation
$$ \forall X \in \mathcal{H}_n, \ T_f(X) = \sup_{Y \in \mathcal{H}_n \atop \mathrm{rank}(Y)\leq k} \{ \tr f(Y) + \tr \big(f'(Y)(X-Y)\big)\}.$$ 
We deduce that for any $X,Y\in\mathcal{H}_n$,
$$ T_f(X) - T_f(Y) \leq \tr\big(f'(Z)(X-Y)\big),$$
where $Z = \sum_{i=1}^k \lambda_i u_i u_i^*$, with $\lambda_1 \geq \ldots \geq \lambda_n$ the eigenvalues of $X$ and $u_1,\ldots,u_n$ the associated eigenvectors. If we take $f : x \mapsto  x_+^d$, $x\mapsto x_-^d$ or $x\mapsto x^d$ in the case $d$ is even, and $X \in 4C n B_2$, then we see that $f'(Z)$  is of rank $k$ and $|| f'(Z)|| \leq d|| X||^{d-1} \leq d(4C)^{d-1} n^{d-1}$. As by definition, $f_k$ is a combination of at most two functions $T_f$ associated with the functions $x\mapsto x_+^d$, $x\mapsto x_-^d$ or $x\mapsto x^d$, we get that for any $X,Y \in 4C B_2$, 
\begin{equation} \label{difftr}f_k(X) - f_k(Y) \leq \sup_{H \in W} \tr H(X-Y),\end{equation}
where 
\begin{equation} \label{defV} W = \{ H \in \mathcal{H}_n : \rk (H)\leq 2k, ||H||\leq c_0 n^{d-1} \},\end{equation}
where $c_0$ is some positive constant depending on  $C$.

Note that by \eqref{inclu}, the diameter of the level set $K$ is bounded by $8Cn$, and by \eqref{difftr}, the Lipschitz constant of $f_k$ on $K$  is only polynomial in $n$. By Theorem \ref{NL}, we get
\begin{align}
\log\PP\big( f_k(X) \notin V_{\delta}( \{ J \leq rn^{1+\frac{2}{d}} \}) \big) &\leq  - r n^{1+ \frac{2}{d}} \nonumber \\
& \label{upptr}+\log N(8 n c W, \delta B_{2}) + O(\log n),
\end{align}
where $N(8ncW,\delta B_{2})$ denotes the covering number of $8ncW$ by Hilbert-Schmidt balls of radius $\delta$.
It now remains to compute the covering numbers of $W$. This is the object of the following lemma. 
\begin{Lem}\label{complexity}Let $k\in\{1,\ldots,n\}$.
Define the set
$$\mathcal{M}_k = \{ Y \in \mathcal{H}_n : \rk (Y)\leq k, ||Y||\leq 1 \}.$$
For any $\eps\in (0,1)$, let $N(\mathcal{M}_k, \eps B_{2})$ be the covering number of $\mathcal{M}_k$ by Hilbert-Schmidt balls of radius $\eps$. Then, 
$$\log N(\mathcal{M}_k, \eps B_{2}) \leq 2 nk \log \Big(\frac{12k }{\eps}\Big).$$

\end{Lem}

\begin{proof}
Let $Y \in \mathcal{M}_k$ and $Z \in \mathcal{H}_n$. Let us spectrally decompose $Y$, and write $Z$ as,
$$ Y = \sum_{i=1}^k \lambda_i u_i u_i^*,\quad Z = \sum_{i=1}^k \mu_i v_i v_i^*,$$
where $\mu_i$ are real numbers, $v_i$ are unit vectors (not necessarily orthogonal to one another),  $\lambda_1,\ldots,\lambda_k$ are the possible non-zero eigenvalues of $Y$ and $u_1,\ldots,u_k$ the associated eigenvectors. As the $v_i$'s are unit vectors,
$$||Y-Z||_2 \leq \sum_{i=1}^k |\lambda_i- \mu_i| + \sum_{i=1}^k |\lambda_i|. || u_i u_i^* - v_iv_i^*||_2.$$
Since $||Y|| \leq 1$ and $|| uv^*||_2 =||u||_2 ||v||_2$ for any $u,v \in \CC^n$, we get,
\begin{equation}\label{ineqcomplex}  ||Y-Z|| \leq k\max_{1\leq i \leq k} |\lambda_i - \mu_i| + 2k \max_{1\leq i \leq k} || u_i -v_i||_2.\end{equation}
Let $\mathcal{E}$ be a $\eps/2k$-net of $[-1,1]$ and $\mathcal{F}$ a $\eps/4k$-net of the unit sphere $\mathbb{S}^{n-1}$. Then, the set
$$ \mathcal{D}_{\eps} = \big\{ \sum_{i=1}^k \mu_i v_i v_i^* : \mu_i \in \mathcal{E}, \ u_i \in\mathcal{F} \big\},$$
is an $\eps$-net of $\mathcal{M}_k$ due to the inequality \eqref{ineqcomplex}. Clearly, we can find a net $\mathcal{E}$ such that $|\mathcal{E}| \leq 2k/\eps$. Moreover, by \cite[Lemma 1.4.2]{Lamabook}, there exists a net $\mathcal{F}$ such that,
$$ |\mathcal{F}| \leq \Big( \frac{12 k }{\eps}\Big)^n.$$
From the construction of $\mathcal{D}_{\eps}$ we get finally, 
\[\log |\mathcal{D}_{\eps}| \leq 2 nk \log \Big( \frac{12 k}{\eps} \Big).\]
\end{proof}
Coming back to the inequality \eqref{upptr}, we see that we have the inclusion $8CnW \subset  n^{b} \mathcal{M}_{2k}$, for some $b>0$. Thus, by Lemma \ref{complexity}, 
$$ \log N(8C n W, \delta B_{\ell^2}) \leq N( \mathcal{M}_{2k}, \delta n^{-b} B_{\ell^2}) =O\Big(nk \log \Big( \frac{n}{\delta}\Big)\Big).$$
As we chose $k\log n=o(n^{\frac{2}{d}})$, we get
$$\limsup_{n\to +\infty} \frac{1}{n^{1+\frac{2}{d}}} \log\PP\big( f_k(X) \notin V_{\delta}( \{ J \leq rn^{1+\frac{2}{d}} \}) \big) \leq  - r,$$
which implies, because of \eqref{reductionup},
$$\limsup_{n\to +\infty} \frac{1}{n^{1+\frac{2}{d}}} \log\PP\big( f_k(X) \in F\big) \leq -r.$$
As this inequality is true for any $r<\inf_F J_+$, we obtain the large deviation upper bound of Theorem \ref{LDPtrb}.

\subsection{Large deviation lower bound}
In this section we will use Lemma \ref{borneinfdet} to prove the large deviation lower bound for $(\frac{1}{n} \tr(X/\sqrt{n})^d)_{n\in\NN}$. One of the difficulty in the derivation of the lower bound is that the domain of $\Lambda^*$ may be different from $\mathcal{H}_n$. To bypass this problem, we will add a small Gaussian noise, as in the proof of Cramer's theorem (see \cite[Proof of Theorem 2.2.30]{DZ}). 

\subsubsection{Regularization by Gaussian noise}Recall that we denote by $\mathcal{H}_n^{(\beta)}$ the set of Hermitian matrices when $\beta=2 $, and symmetric matrices when $\beta=1$, of size $n$.
We say that $\Gamma$ belongs to the \textit{Gaussian orthogonal ensemble} (GOE) when $\beta=1$, or to the \textit{Gaussian unitary ensemble} (GUE) when $\beta =2$, if it distributed according to the law,
$$ \frac{1}{Z_n^{(\beta)} }e^{-\frac{\beta}{4} \tr H^2 } d \ell_n^{(\beta)}(H),$$
where $\ell_n^{(\beta)}$ is the Lebesgue measure on $\mathcal{H}_n^{(\beta)}$. 
\begin{Lem}Let $X$ be a Wigner matrix satisfying the convex concentration property \ref{concconv}. 
Let $\Gamma$ be a GOE matrix if $\beta=1$ or a GUE matrix if $\beta=2$. For any $t>0$, 
$$ \lim_{\eps \to 0 } \limsup_{n\to +\infty} \frac{1}{n^{1+\frac{2}{d}}} \log \PP\big( \big|  \tr (X/\sqrt{n})^d - \tr\big((X+\eps \Gamma)/\sqrt{n}\big)^d \big|>t n \big) = - \infty.$$
\end{Lem}
\begin{proof}For $\sigma \in \{-1,1\}$ and $Y\in\mathcal{H}_n$, we write $Y_{\sigma} = Y_+$ if $\sigma =+$ and $Y_{\sigma} = Y_-$ if $\sigma = -$.
With this notation, it is sufficient to prove that for $\sigma\in\{-1,1\}$ and any $t>0$,
$$ \lim_{\eps \to 0 } \limsup_{n\to +\infty} \frac{1}{n^{1+\frac{2}{d}}} \log \PP\big( \big|  \tr (X_\sigma/\sqrt{n})^d - \tr\big((X+\eps \Gamma)_\sigma/\sqrt{n}\big)^d \big|>t n \big) = - \infty.$$
Fix $\sigma \in \{-1,1\}$. 
We will first show that for any $\eps\geq 0$, the sequence $(n^{-1}\tr|X+\eps \Gamma|^d)_{n\in \NN}$ is exponentially tight at the scale $n^{1+\frac{2}{d}}$. This will allow us to reduce ourselves to prove that  $n^{-1/d}||(X+\eps\Gamma)_\sigma/\sqrt{n}||_d$ is an exponentially good approximation of $n^{-1/d}||X_\sigma/\sqrt{n}||_d$, where $||\ ||_d$ denotes the $d^{\text{th}}$ Schatten norm,
\begin{equation} \label{defSchatten} || \ ||_d :  Y\in \mathcal{H}_n \mapsto \big( \tr|Y|^d\big)^{1/d}.\end{equation}
Since $|| \ ||_d$ is a convex $1$-Lipschitz function with respect to the Hilbert-Schmidt norm, we obtain by applying the convex concentration property \ref{concconv} that for any $t>0$,
 \begin{equation} \label{tightSchatten} \PP\big (  || X/\sqrt{n} ||_d - \EE || X/\sqrt{n}||_d  > t n^{1/d} \big) \leq c^{-1}e^{- c t^2 n^{1+\frac{2}{d}}}.\end{equation}  
But, using Cauchy-Schwarz inequality and Jensen's inequality we have,
$$  \EE\big( n^{-1/d} || X/\sqrt{n}||_d \big) \leq \EE \big( \frac{1}{n} \tr (X/\sqrt{n})^{2d} \big)^{\frac{1}{2d}}\big) \leq  \big( \EE \frac{1}{n} \tr (X/\sqrt{n})^{2d} \big)^{\frac{1}{2d}}\big).$$
By Wigner's theorem \cite[Lemma 2.1.6]{AGZ}, we get that $  \EE || X/\sqrt{n}||_d  =O(n^{1/d})$. Therefore, we can deduce from \eqref{tightSchatten} that $(\frac{1}{n}\tr|X / \sqrt{n}|^d)_{n\in \NN}$ is exponentially tight at the scale $n^{1+\frac{2}{d}}$. Using further the triangle inequality for the $d^{\text{th}}$ Schatten norm, and the fact that $\Gamma$ also satisfies the convex concentration property by \cite[Theorem 5.2, 4.3]{Ledouxmono}, we deduce that for any $\eps\geq 0$, $(\frac{1}{n}\tr|(X+\eps \Gamma)/\sqrt{n}|^d)_{n\in \NN}$ is exponentially tight at the scale $n^{1+\frac{2}{d}}$. Therefore, it suffices to show that for arbitrary large but fixed $\tau>0$ and any $t>0$,
$$ \lim_{\eps \to 0 } \limsup_{n\to +\infty} \frac{1}{n^{1+\frac{2}{d}}} \log \PP_{A_\tau}\big( \big|  \mu_{X/\sqrt{n}}(x^d_\sigma)- \mu_{(X+\eps \Gamma)/\sqrt{n}}(x^d_\sigma) \big|>t \big) = - \infty,$$
where $\PP_{A_\tau}$ denotes the measure $\PP( . \cap A_\tau)$, and $A_\tau$ is the event,
$$ A_\tau = \big\{ \mu_{X/\sqrt{n}}(|x|^d) \leq \tau, \ \mu_{(X+\eps \Gamma)/\sqrt{n}}(|x|^d) \leq \tau \big\}.$$
Using the uniform continuity of $x\mapsto |x|^{1/d}$ on compact sets, we deduce that it is enough to show for any $\delta>0$,
$$ \lim_{\eps \to 0 } \limsup_{n\to +\infty} \frac{1}{n^{1+\frac{2}{d}}} \log \PP\big( \big|\big(\mu_{X/\sqrt{n}}(x^d_\sigma)\big)^{\frac{1}{d}}- \big(\mu_{(X+\eps \Gamma)/\sqrt{n}}(x^d_\sigma)\big)^{\frac{1}{d}}\big|>\delta \big) = - \infty.$$
But, by Minkowski's inequality, 
$$ \big|\big(\mu_{X/\sqrt{n}}(x^d_\sigma)\big)^{\frac{1}{d}}- \big(\mu_{(X+\eps \Gamma)/\sqrt{n}}(x^d_\sigma)\big)^{\frac{1}{d}} \big| \leq \Big(\frac{1}{n} \sum_{i=1}^n \big|(\lambda_i)_\sigma - (\mu_i)_\sigma \big|^d \Big)^{\frac{1}{d}},$$
where $\lambda_1\geq\ldots\geq \lambda_n$ and $\mu_1\geq \ldots\geq \mu_n$ are the eigenvalues of respectively  $X/\sqrt{n}$ and $(X+\eps \Gamma)/\sqrt{n}$. But,
$$\sum_{i=1}^n \big|(\lambda_i)_\sigma - (\mu_i)_\sigma \big|^d  \leq  \sum_{i=1}^n \big|\lambda_i - \mu_i \big|^d,$$
and  by Lidskii's theorem  (see \cite[Theorem III.4.1]{Bhatia}),
$$ \frac{1}{n}\sum_{i=1}^n \big|\lambda_i - \mu_i \big|^d \leq \mu_{\eps \Gamma/\sqrt{n}}(|x|^d).$$
Thus it actually suffices to prove that for any $\delta > 0$,
$$ \lim_{\eps \to 0 } \limsup_{n\to +\infty} \frac{1}{n^{1+\frac{2}{d}}} \log \PP\big( \mu_{\eps \Gamma/\sqrt{n}}(|x|^d) >\delta \big) = - \infty,$$
But $(\frac{1}{n}\tr|\Gamma/\sqrt{n}|^d)$ is exponentially tight at the scale $n^{1+\frac{2}{d}}$, therefore we get finally the claim.
\end{proof}
\subsubsection{A tilting strategy}
By \cite[Theorem 4.2.16]{DZ}, it is sufficient  to prove the large deviations lower bound for the sequence $(\frac{1}{n}\tr ((X+\eps \Gamma)/\sqrt{n})^d)_{n\in\NN}$ for any $\eps>0$ small enough. We denote by $\Lambda^*_\eps$ the Legendre transform of the logarithmic Laplace transform $\Lambda_\eps$ of $X+\eps \Gamma$, that is,
$$\forall H \in \mathcal{H}_n^{(\beta)}, \ \Lambda_\eps(H) = \log \EE e^{\tr [(X+\eps \Gamma)]} = \Lambda(H) + \frac{\eps^2 \tr H^2}{\beta}.$$
Adding this small Gaussian noise yields that the domain of $\Lambda^*_{\eps}$ is $\mathcal{H}_n^{(\beta)}$. We note also for future record that, 
\begin{equation} \label{lowb}\forall Y \in \mathcal{H}_n^{(\beta)}, \  \Lambda^*_{\eps}(Y) \leq \Lambda^*(Y).\end{equation}
Let us now proceed with the proof of the large deviation lower bound. Denote by $Z = X + \eps \Gamma$. Let $x \in \RR$, such that $I_-(x)<+\infty$. This means, from the definition of $I_-$ in Theorem \ref{LDPtrb}, that we can find a sequence $Y \in \mathcal{H}_n^{(\beta)}$ such that,
\begin{equation} \label{conv} \lim_{n\to +\infty} \mu_{sc}(x^d) + \frac{1}{n} \tr(Y/\sqrt{n})^d \underset{n\to +\infty}{\longrightarrow} x, \quad \lim_{n\to +\infty} \frac{\Lambda^*(Y)}{n^{1+\frac{2}{d}}} = I_-(x).\end{equation}
Let $\delta>0$. For $n$ large enough, we have
$$ \PP\Big( \frac{1}{n}\tr(Z /\sqrt{n})^d \in B(x,2\delta)\Big) \geq \PP\Big( \frac{1}{n} \tr(Z/\sqrt{n})^d \in B(s_n, \delta) \Big),$$
where $s_n = \mu_{sc}(x^d) + \frac{1}{n} \tr(Y/\sqrt{n})^d$. Let $E$ denote the event,
$$ E = \big\{A \in \mathcal{H}_n^{(\beta)} : \frac{1}{n}\tr(A/\sqrt{n})^d \in B(s_n,\delta) \big\}.$$
Using  Lemma \ref{borneinfdet} and \eqref{lowb}, we have,
$$ \PP( Z  \in E) \geq e^{-\Lambda^*(Y)} \PP_{Y}(Z \in E) \exp\Big( - \frac{1}{\PP_{Y}(Z\in E)^{1/2}} \langle H, \nabla^2 \Lambda_{\eps}(H).H\rangle^{1/2}\Big),$$
where $H = \nabla \Lambda^*_{\eps}(Y)$, and under $\PP_{Y}$, $Z$ follows the tilted law,
$$ e^{\tr (HZ) - \Lambda_{\eps}(H)} d\PP(Z).$$
To obtain the lower bound, we will prove on one hand,
\begin{equation} \label{claim1}\langle H, \nabla^2 \Lambda_{\eps}(H).H\rangle = O(\Lambda^*(Y)).\end{equation}
and on the other hand,
that 
\begin{equation} \label{claim2} \PP_{Y}(Z\in V) \underset{n\to+\infty}{\longrightarrow} 1,\end{equation}
Provided these two claims hold, we get, as $\Lambda^*(Y) = O(n^{1+\frac{2}{d}})$,
$$ \liminf_{n\to +\infty} \frac{1}{n^{1+\frac{2}{d}}} \log \PP\Big( \frac{1}{n}\tr(Z/\sqrt{n})^d \in B(x,\delta)\Big) \geq - I_-(x),$$
which gives the lower bound of Theorem \ref{LDPtrb}.

Let us now prove \eqref{claim1} and \eqref{claim2}. For the first claim, note that for any $A\in\mathcal{H}_n^{(\beta)}$, 
\begin{equation}\label{secondderiv}\nabla^2\Lambda_\eps(A)  = \nabla^2 \Lambda(A) + \frac{2\eps^2}{\beta}Id,\end{equation}
and for any $H \in\mathcal{H}_n^{(\beta)}$, 
$$\tr \big( H \nabla^2 \Lambda(A). H\big)  = \sum_{i=1}^n\frac{\partial^2 \Lambda_{1,1}(A_{i,i})}{\partial A_{1,1}^2} H_{i,i}^2 +  \sum_{i\neq j}\langle H_{i,j},\nabla^2 \Lambda_{1,2}(A_{i,j}). H_{i,j}\rangle,$$
where $\langle . , .\rangle$ denotes the scalar product in $\CC^2$. As we assumed that the second derivatives of $\Lambda_{1,1}$ and $\Lambda_{1,2}$ are bounded, we deduce that there exists $r>0$ such that for any $A \in \mathcal{H}_n^{(\beta)}$, $\nabla^2 \Lambda_\eps (A) \leq r$, for the matrix order. 
Thus, it is sufficient to prove that for any $Y\in \mathcal{H}_n^{(\beta)}$,
\begin{equation} \label{compKH} \tr H^2 =O( \Lambda^*(Y)),\end{equation}
where $H=\nabla \Lambda_\eps^*(Y)$. From \eqref{secondderiv} we have in particular $\nabla^2 \Lambda_\eps(A)\geq \frac{2\eps^2}{\beta}$ for the matrix order. Therefore, denoting by $\Lambda_{\eps, (i,j)}$ the logarithmic Laplace transform of $X_{i,j}+\eps \Gamma_{i,j}$ for any $i,j \in \{1,\ldots,n\}$, and integrating the inequalities,
$$\forall A \in \mathcal{H}_n^{(\beta)}, \  \forall i<j, \ \Lambda_{\eps, (i,i)}''(A_{i,i}) \geq \frac{2\eps^2}{\beta}, \ \nabla^2\Lambda_{\eps,(i,j)}(A_{i,j}) \geq \frac{2\eps^2}{\beta},$$
(where the last inequality is meant in the matrix order when $\beta=2$), knowing that $\nabla \Lambda_\eps(0)=0$, we obtain
$$ || \nabla \Lambda_\eps(H) ||_2 \geq \eps^2 || H||_2.$$
From \eqref{subG}, we deduce that,
$$  \Lambda^*(\nabla \Lambda_\eps(H)) \geq \frac{1}{2C^2} \tr \big(\nabla \Lambda_\eps(H)\big)^2\geq \frac{\eps^4}{2C^2} \tr H^2,$$  
which proves the estimate \eqref{compKH} since by definition $\nabla \Lambda_\eps(H) = Y$.

For the second claim \eqref{claim2}, we observe that under $\PP_{Y}$, $Z$ is a random Hermitian matrix with independent entries up to the symmetry with mean $\nabla \Lambda_\eps(H) = Y$. Moreover, if $\sigma_{i,j}^2 = \EE_H| Z_{i,j} - \EE Z_{i,j}|^2$ is the variance of the $(i,j)^{\text{th}}$ coordinate for $i<j$, then
\begin{align*}
|1+\eps^2-\sigma_{i,j}^2|& =  | \tr \nabla^2\Lambda_{\eps,(i,j)}(0) - \tr \nabla^2\Lambda_{\eps,(i,j)}(H_{i,j})|\\
& = | \tr \nabla^2\Lambda_{i,j}(0) - \tr \nabla^2\Lambda_{i,j}(H_{i,j})|,\end{align*}
and
$$ \sigma_{i,i}^2 = \Lambda_{i,i}''(H_{i,i}) + \frac{2\eps^2}{\beta}.$$
As we assumed the second and third derivatives of $\Lambda_{i,j}$ to be uniformly bounded, we deduce that $|1+\eps^2-\sigma_{i,j}^2| =  O( |H_{i,j}|)$, and $\sigma_{i,i}^2 =O(1)$. Thus, using the fact $\sqrt{x+y} \leq \sqrt{x}+\sqrt{y}$ for $x,y\geq 0$, we get 
\begin{equation} \label{controlvar} \sum_{i,j} (\sqrt{1+\eps^2}-\sigma_{i,j})^2 \leq \sum_{i,j} |1+\eps^2-\sigma_{i,j}^2|\leq O\big( n (|| H||_2\vee 1)\big),\end{equation}
where we further used Cauchy-Schwarz inequality. Therefore, from \eqref{compKH}, we get
$$\sum_{i,j} (\sqrt{1+\eps^2}-\sigma_{i,j})^2 = O\big( n\big( \Lambda^*(Y)^{1/2}\vee 1\big)\big) .$$
But $\Lambda^*(Y) = O(n^{1+\frac{2}{d}})$ by \eqref{conv}. Thus, we finally obtain,
$$\sum_{i,j} (\sqrt{1+\eps^2}-\sigma_{i,j})^2 = O( n^{\frac{3}{2} + \frac{1}{d}}) = o(n^2).$$
Thus, to show the claim  \eqref{claim2}, it remains to prove the following lemma, which is where the assumptions we made on the boundedness of the derivatives of $\Lambda$ play further their roles.

\begin{Lem}Let $d \geq 3$ and $\ell$ even such that $\ell> d$.
Let $X$ be a random Hermitian matrix such that $(X_{i,j})_{i\leq j}$ are independent. Assume that the $\ell^{\text{th}}$ moments of $X_{i,j}$ are uniformly bounded, 
$$ \tr (\EE X)^2 = O( n^{1+\frac{2}{d}}), \text{ and } \sum_{i,j} (1-\sigma_{i,j})^2 =o(n^2),$$
where $\sigma_{i,j}^2 = \EE |X_{i,j}-\EE X_{i,j}|^2$.
Then,
$$ \big| \frac{1}{n} \tr(X/\sqrt{n})^d - \mu_{sc}(x^d) - \frac{1}{n}\tr(\EE X/\sqrt{n})^d\big| \underset{n\to +\infty}{\longrightarrow} 0,$$
in probability.
\end{Lem}

\begin{proof}Let $\hat X = X - \EE X$.
From \cite[Lemma 2.1 (9)]{LDPtr}, we know that expanding $\tr (\hat X+\EE X)^d$, we obtain by bounding the mixed terms using H\"older's inequality,
$$\big| \tr X^d - \tr {\hat X}^d - \tr (\EE X)^d \big| \leq 2^d \max_{1\leq k\leq d-1} || \hat X||_{d+1}^{k} || \EE X ||_2^{d-k},$$
where $|| \ ||_m$ denotes the $m^{\text{th}}$ Schatten norm, which is defined in \eqref{defSchatten}.
As $\ell> d$,  we have $ ||\hat X||_{d+1} \leq || \hat X ||_\ell$. Using the assumption  that $\tr (\EE X)^2 = O( n^{1+\frac{2}{d}})$, we get,
$$\big| \tr X^d - \tr {\hat X}^d - \tr (\EE X)^d \big| =  O\Big(  \max_{1\leq k\leq d-1} || \hat X ||_\ell^{k} n^{(\frac{1}{2}+\frac{1}{d})(d-k)}\Big).$$
But as $\ell$ is even,
\begin{equation} \label{bound}\EE \tr {\hat X}^\ell  = O( n^{1+\frac{\ell}{2}}),\end{equation}
by expanding the trace and using the fact that the $\ell^{\text{th}}$ moments of the entries $X$ are uniformly bounded.  By Markov's inequality, we deduce that 
$$ \PP( || \hat X||_\ell \geq (\log n) n^{\frac{1}{\ell}+\frac{1}{2}})\underset{n\to +\infty}{\longrightarrow} 0.$$
 Therefore, 
$$\big| \tr X^d - \tr {\hat X}^d - \tr (\EE X)^d \big| =  O\Big( (\log n)^d \max_{1\leq k\leq d-1} n^{(\frac{1}{2}+ \frac{1}{\ell})k} n^{(\frac{1}{2}+\frac{1}{d})(d-k)}\Big),$$
with probability going to $1$.
As $\ell>d$, the maximum on the right-hand side is achieved at $k=1$, so that
$$\big| \tr X^d - \tr {\hat X}^d - \tr (\EE X)^d \big| =  O\Big((\log n)^d n^{1+\frac{d}{2}}  n^{-(\frac{1}{d} -\frac{1}{\ell})}\Big) = o(n^{1+\frac{d}{2}}),$$
with probability going to $1$. It now remains to show that,
$$ \frac{1}{n} \tr ({\hat X}/\sqrt{n})^d \underset{n\to +\infty}{\longrightarrow}  \mu_{sc}(x^d),$$
in probability.  Let $\Sigma = (\sigma_{i,j})_{i,j}$ and $\hat Y$ such that $\hat X = \Sigma \circ  \hat Y$, where $\circ$ denotes the Hadamard product. From Wigner's theorem, we know that in probability, 
$$ \mu_{\hat Y/\sqrt{n}} \underset{n\to +\infty}{\leadsto} \mu_{sc},$$
weakly.  Denote by $\mathcal{W}_2$ the $L^2$-Wasserstein distance on the space $\mathcal{P}(\RR)$ of probability measures on $\RR$, which is defined by,
$$ \forall \mu,\nu \in \mathcal{P}(\RR), \ \mathcal{W}_2(\mu,\nu) = \inf_{\pi} \int |x-y|^2 d\pi(x,y),$$ 
where the infimum runs over all couplings between $\mu$ and $\nu$. From the assumption on the variance profile of $X$, we have by Hoffman-Weilandt inequality (see \cite[Theorem VI.4.1]{Bhatia}),
$$ \EE \mathcal{W}_2(\mu_{\hat X/\sqrt{n}}, \mu_{\hat Y/\sqrt{n}}) \underset{n\to +\infty}{\longrightarrow} 0.$$
Thus, together with the weak convergence of $\mu_{\hat Y/\sqrt{n}}$ towards the semi-cercle law, we deduce that $\mathcal{W}_2(\mu_{\hat X/\sqrt{n}}, \mu_{sc})$ converges to $0$ in probability. But, from \eqref{bound}, we get by Markov's inequality,
$$ \lim_{\tau \to +\infty} \limsup_{n\to +\infty} \PP\big (\mu_{\hat X/\sqrt{n}} (x^\ell) \geq \tau \big) \underset{n\to +\infty}{\longrightarrow} 0.$$
Therefore, we can integrate the limit and deduce that
$$ \frac{1}{n} \tr ({\hat X}/\sqrt{n})^d = \mu_{sc}(x^d) + o(1),$$
in probability.

\end{proof}

\subsection{Computation of the rate function}\label{comp}
In this section we give a proof of Corollary \ref{sharpsubGtr}. We first prove a general upper bound on the lower bound rate function $I_-$. This bound corresponds to the strategy of increasing (or decreasing) uniformly the mean of the off-diagonal entries.
\begin{Lem}\label{upperboundrate}
Assume $X$ is a Wigner matrix such that in the case $\beta=2$, $(\Re X_{1,2}, \Im X_{1,2})$ are each of variance $1/2$. Then,
$$ \forall x\geq 0, \ \limsup_{n\to +\infty} I_n(x) \leq \frac{\beta}{4} x^{\frac{2}{d}},$$
where
$$ I_n(x) = \inf \big\{ \frac{\Lambda^*(Y)}{n^{1+\frac{2}{d}}} : \frac{1}{n} \tr(Y/\sqrt{n})^d  =x , Y \in\mathcal{H}_n^{(\beta)}\big\}.$$ 
 In particular,
$$  I_-(x) \leq  \begin{cases}
\frac{\beta}{4} \big( x-\mu_{sc}(x^d)\big)^{\frac{2}{d}} & \text{ if $d$ is even and $x \geq \mu_{sc}(x^d)$,}\\
 \frac{\beta}{4} |x|^{\frac{2}{d}}& \text{ if $d$ is odd and $x\in \RR$}.
\end{cases}$$
\end{Lem}
\begin{proof}Let $x\geq 0$. 
Let  $y = \Big( \frac{x}{(n-1)^d+(n-1)} \Big)^{\frac{1}{d}} n^{\frac{1}{2} + \frac{1}{d}}$, and define \begin{equation*}
Y
= \left(
     \raisebox{0.5\depth}{%
       \xymatrixcolsep{1ex}%
       \xymatrixrowsep{1ex}%
       \xymatrix{
        0 \ar @{.}[dddrrr]& y \ar @{.}[rr] \ar @{.}[ddrr] &   &y \ar @{.}[dd]  \\
         y \ar@{.}[dd] \ar@{.}[ddrr] & & & \\
         &&&y \\
         y \ar@{.}[rr] &  &y& 0
       }%
     }
   \right)\in  \mathcal{H}_n. \end{equation*}
Then, $\tr (Y^d) =xn^{1+\frac{d}{2}}$, and
$$ \Lambda^*(Y) = \frac{n(n-1)}{2} \Lambda^*_{\Re X_{1,2}}(y).$$

But $y \sim n^{\frac{1}{d}-\frac{1}{2}} x^{\frac{1}{d}}$. As $\Re X_{1,2}$ has variance $1/\beta$, we claim that the Taylor expansion of $\Lambda^*$ around $0$ gives $\Lambda^*_{\Re X_{1,2}} (y) = \frac{\beta y^2}{2} + o(y^2)$. Indeed, if we let $ g = \Lambda^*_{\Re X_{1,2}}$, then one can derive from the identity $ (g^*)' \circ g' =Id$ the fact that
$$\forall \lambda \in \RR, \   (g^*)''( g' (\lambda )) = \frac{1}{g''(\lambda)}.$$
Since $\Re X_{1,2}$ is centered and has variance $1/\beta$. We have $g'(0) = 0$ and $g''(0) =1/\beta$. Thus, $(g^*)''(0)=\beta$ and therefore $\Lambda^*_{\Re X_{1,2}} (y) = \frac{\beta y^2}{2} + o(y^2)$ as $y \to0$. Thus,
$$ \Lambda^*(Y) = \frac{n(n-1)}{2}\Big( \frac{\beta}
{2} x^{\frac{2}{d}}n^{\frac{2}{d}-1} + o(n^{\frac{2}{d}-1})\Big) = \big(\frac{\beta}{4} x^{\frac{2}{d}} + o(1)\big) n^{1+\frac{2}{d}},$$
which ends the proof.
\end{proof}

\begin{Rem}
One can produce a second upper bound on $I_-$ which corresponds this time to the strategy of having one very large entry on the diagonal (for example). It reads,
$$ \forall x \geq 0, \ I_-(x+\mu_{sc}(x^d)) \leq \limsup_{n\to +\infty} \frac{\Lambda^*_{1,1}(x^{\frac{1}{d}} n^{\frac{1}{d}+\frac{1}{2}})}{n^{1+\frac{2}{d}}},$$
which is nontrivial if $\Lambda^*_{1,1}$ has a quadratic behavior at infinity. In particular, this bound shows that one can have a rate function very different from the one of the GOE or GUE. Indeed, if $\Lambda_{1,1}(\lambda) \simeq_{+\infty} \frac{B}{2}\lambda^2$, then $\Lambda^*_{1,1}(x)\simeq_{+\infty} \frac{1}{2B}x^2$, so that,
$$ \forall x \geq 0, \ I_-(x+\mu_{sc}(x^d)) \leq \frac{1}{2B}x^{\frac{2}{d}}.$$
This bound is somewhat related to the large deviations of the traces of Wigner matrices with entries having super-Gaussian tails (see \cite{LDPtr} for more details), where a heavy-tail phenomenon arises, in the sense that only the large entries of the matrix are controlling the large deviation of the trace.

\end{Rem}

In the case where the entries of $X$ have sharp sub-Gaussian tails, we show in the following proposition that the upper bound on $I_-$ proved in Lemma \ref{upperboundrate} is also a lower bound on $I_+$. By Theorem \ref{LDPtrb}, this fact immediately implies that a full LDP holds for the traces of this class of Wigner matrices, as announced in Corollary \ref{sharpsubGtr}.

\begin{Pro}
Assume $X$ is a Wigner matrix whose entries have sharp sub-Gaussian tail. Assume further that $\EE X_{1,1}^2 \leq \beta/2$, and  in the case  $\beta=2$ that $(\Re X_{1,2}, \Im X_{1,2})$ are each of variance $1/2$.
Then,
$$ I_+ = I_- = J_d,$$
where $J_d$ is defined in Corollary \ref{sharpsubGtr}, and where for any $x\in\RR$,
$$ I_+(x) =  \sup_{\delta>0}\liminf_{n\to +\infty} I_{n,\delta}(x),$$
$$ I_-(x) =\sup_{\delta>0}\limsup_{n\to +\infty} I_{n,\delta}(x),$$
and for any $n\in \NN$, $\delta>0$,
$$ I_{n,\delta}(x) = \inf \Big\{ \frac{\Lambda^*(Y)}{n^{1+\frac{2}{d}}} : |\frac{1}{n} \tr (Y/\sqrt{n})^d + \mu_{sc}(x^d) - x| <\delta, Y \in \mathcal{H}_n \Big\}.$$ 
\end{Pro}

\begin{proof}
Since the entries of $X$ have sharp sub-Gaussian tails, we deduce from our assumptions on the covariance structure that for any $H\in \mathcal{H}_n^{(\beta)}$, $\Lambda(H) \leq (1/\beta) \tr H^2$. Therefore we have for any $Y \in \mathcal{H}_n^{(\beta)}$,
$$ \Lambda^*(Y) \geq \frac{\beta}{4} \tr Y^2.$$
But, $| \tr( Y^d) | \leq \tr(|Y|^d) \leq (\tr (Y^2))^{d/2}$, therefore, 
$$ \Lambda^*(Y) \geq \frac{\beta}{4} |\tr (Y^d) |^{\frac{2}{d}}.$$
Thus, for any $x\in\RR$,
$$ I_{+}(x) \geq  \inf \Big\{ \frac{\beta}{4} |y|^{\frac{2}{d}} : y + \mu_{sc}(x^d) = x ,y \in \RR\Big\}.$$ 
In the case $d$ is even, we see that the function on the right-hand side is infinite if $x <\mu_{\sc}(x^d)$ and otherwise equal to $(\beta/4) (x-\mu_{sc}(x^d))^{2/d}$. If $d$ is odd, then one obtains $ I_+(x) \geq (\beta/4) |x|^{2/d}$. In conclusion, we have the inequality $I_+ \geq J_d$. But, by Lemma \ref{upperboundrate}, $I_- \leq  J_d$, which completes the proof.

\end{proof}

\section{Upper tail of cycle counts in sparse  Erd\H{o}s--R\'enyi graphs}

In this section, we will give a proof of Proposition \ref{cycleup}. As for the traces of Wigner matrices, the main idea is to reduce the complexity by showing that we can replace the full trace by a truncated version, involving only the eigenvalues at the edges of the spectrum.  Then, using standard complexity computations, we apply Proposition \ref{NL} to obtain the desired upper bound.
\subsection{A truncation argument}We recall that for any $Y \in \mathcal{H}_n$, we denote by $\lambda_1(Y),\ldots,\lambda_n(Y)$ its eigenvalues in non-decreasing order and for any $k \in \{1,\ldots,n\}$, we denote by $\tr_{[k]} Y$ the truncated trace:
$$ \tr_{[k]} Y = \sum_{i=1}^k \lambda_i(Y).$$
We define $g_k$ the function,
$$\forall Y \in \mathcal{H}_n, \  g_k(Y) = \begin{cases}
\tr_{[k]} (Y^d )& \text{ if  $d$ is even,}\\
\tr_{[k]} (Y_+^d) - \tr_{[k]} (Y_-^d) & \text{ if  $d$ is odd.}
\end{cases}$$

Building on the concentration inequalities we proved in \S \ref{conc}, we will show the following exponential equivalent.
\begin{Lem}\label{equiER}Denote by $v_n=n^2 p^2 \log (1/p)$. Assume $\log(1/p) = o(np^2)$. Let $k \in \{1,\ldots,n\}$ such that $  \log^4(1/p) = o(k)$, and in the case $d=3$ such that moreover $k =o((np)^{3/2})$. For any $\delta>0$,
$$ \lim_{n\to +\infty} \frac{1}{v_n} \log \PP\big( \tr (X^d) - g_k(X)>\delta p^dn^d\big) = -\infty.$$ 
\end{Lem}
\begin{proof}
The proof will be slightly more involved than its counterpart for Wigner matrices which we proved in Lemma \ref{equivexpo}. We will actually need here a ``slicing argument'' as the one used in the proof the large deviations of the moments of $\beta$-ensembles (see \cite[Proposition 3.5]{LDPtr}).
In a first step, we will prove that for any $\delta>0$, 
\begin{equation} \label{firststep} \lim_{n\to +\infty} \frac{1}{v_n} \log \PP\big( Z - \EE Z>\delta p^dn^d\big) = -\infty,\end{equation}
where $Z = \tr X^d - g_k(X)$.
We will only consider the case where $d$ is odd. In order to show \eqref{firststep}, we see that it suffices to prove that for $\sigma \in \{ -,+\}$, and for any $\delta>0$,
\begin{equation} \label{concen} \lim_{n\to +\infty} \frac{1}{v_n} \log \PP\big( \big|\sum_{i=k+1}^n \lambda_i^d - \EE \sum_{i=k+1}^n \lambda_i^d \big|>\delta p^dn^d\big) = -\infty,\end{equation}
where we use $\lambda_i$ as a shorthand for $\lambda_i(X_\sigma)$ for any $i\in\{1,\ldots,n\}$. Fix $\sigma \in \{ -,+\}$.
Let $M>0$, and define $f_M$ the truncated power function, by $f_M(x) = x^d$ for $x\in [0,M]$, and for $x>M$,
$$f_M(x) =
dM^{d-1}(x-M) +M^d.$$
Note that $f_M$ is by construction convex and $dM^{d-1}$-Lipschitz on $\RR_+$. 
By \cite[Theorem 8.6]{BLM}, we know that $X$ satisfies the convex concentration property as defined in \ref{concconv} with constants $\kappa = 1/2$ and $C=2$. Observe that we can write,
$$ \sum_{i=k+1}^n f_M(\lambda_i) = \tr f_M(X_\sigma) - \tr_{[k]}f_M(X_\sigma).$$
Applying Proposition \ref{conckconv} to the function $x \mapsto f_M(x_\sigma)$ which is convex and $dM^{d-1}$-Lipschitz, we get 
\begin{equation} \label{conctruncER} \PP\Big(\big| \sum_{i=k+1}^{n} f_M(\lambda_i)  - \EE \sum_{i=k+1}^{n} f_M(\lambda_i)  | >\delta n^dp^d \Big) \leq 4\exp\Big( - \frac{n^{2d-1}p^{2d} \delta^2}{8 d^2M^{2(d-1)}}\Big).\end{equation}
Let $\alpha_ n$ to be a sequence going to $+\infty$ such that,
\begin{equation}\label{defalpha} \alpha_n^2 = o\Big( \frac{\sqrt{n} p}{(\log  (1/p))^{\frac{1}{d-1}}}\Big),\end{equation}
which is possible since we assumed $\log(1/p)= o(np^2)$.
We set
 \begin{equation} \label{defM} M = \sqrt{np}\alpha_n.\end{equation} 
 Re-writing \eqref{conctruncER} with this choice of $M$, we have
\begin{equation} \label{alpha} \PP\Big( \big|\sum_{i=k+1}^{n} f_M(\lambda_i)-\EE \sum_{i=k+1}^{n} f_M(\lambda_i) \big| >\delta n^dp^d \Big) \leq 4\exp\Big( - \frac{ n^{d}p^{d+1}\delta^2}{ 8 d^2\alpha_n^{2(d-1)} }\Big).\end{equation}
 But,
\begin{Lem}\label{expectrun} Assume $\log n = o(np)$. There exists a constant $C_0>0$ such that for any $M\geq C_0 \sqrt{np}$, and $k\geq 1$.
$$
\big| \EE \sum_{i=k+1}^{n} ( \lambda_i^d-f_M(\lambda_i))\big| =o(n^dp^d).
$$
\end{Lem}
\begin{proof}
We have,
$$ \big |  \sum_{i=k+1}^n \lambda_i^d- \sum_{i=k+1}^{n} f_M(\lambda_i)\big | \leq \sum_{i=k+1}^n \Car_{\lambda_i\geq M} \lambda_i^d.$$
Using the fact that $\lambda_i\leq n$ since the entries of $X$ are bounded by $1$, we deduce that,
\begin{equation} \label{truncex} \big |  \EE \sum_{i=k+1}^{n} (  \lambda_i^d-f_M(\lambda_i))\big |\leq n^{d+1} \PP\big( \max(- \lambda_n(X), \lambda_2(X)) \geq M\big).\end{equation}
Let $u\in\RR^n$ be the vector $(1,1,\ldots,1)$, and and let $\hat X = X - p u u^*$. By Weyl's inequalities (see \cite[Theorem III.2.1]{Bhatia}), we have 
$$ \max(- \lambda_n(X), \lambda_2(X)) \leq || \hat X||.$$
Since $\log n =o(np)$, we know from \cite[Example 4.10]{LHY}, that $\EE || \hat X|| =O(\sqrt{np})$. As noted before, $\hat X$ has the convex concentration property with constants $\kappa =1/2$ and $C=2$. Using the fact that the spectral radius of a Hermitian matrix is a convex and $1$-Lipschitz function with respect to the Hilbert-Schmidt norm, we deduce that there exist $c_0, \alpha>0$ such that for any $c\geq c_0 \sqrt{np}$,
\begin{equation} \label{tailsp}\PP( || \hat X || \geq c) \leq 2e^{-\alpha c^2}.\end{equation}
As we assumed $\log n =o(np)$, this bound is exponentially small, which, in view of \eqref{truncex} ends proof of the lemma.
\end{proof}
By Lemma \ref{expectrun}, we deduce that for $n$ large enough, we have
 $$ |\EE \sum_{i=k+1}^n\big( \lambda_i^d -  f_M(\lambda_i)\big)| \leq \delta n^dp^d.$$
Therefore, by \eqref{alpha}, we have for $n$ large enough,
$$  \PP\Big( \big|\sum_{i=k+1}^{n} f_M(\lambda_i)-\EE \sum_{i=k+1}^{n}  \lambda_i^d \big| >2\delta n^dp^d \Big) \leq 4\exp\Big( - \frac{\delta^2 n^{d}p^{d+1}}{ 8 d^2\alpha_n^{2(d-1)} }\Big).$$
Since $d\geq 3$, we have $\frac{d-2}{d-1} \leq \frac{1}{2}$. Thus, from the choice of $\alpha_n$ made in \eqref{defalpha},
$$ \alpha_n^2 = o\Big( \frac{n^{\frac{d-2}{d-1}} p}{(\log  (1/p))^{\frac{1}{d-1}}}\Big),$$ so that
\begin{equation} \label{expoequitrunER}\lim_{n\to +\infty} \frac{1}{v_n} \log \PP\Big( |\sum_{i=k+1}^{n} f_M(\lambda_i)- \EE \sum_{i=k+1}^n \lambda_i^d| >2\delta n^dp^d \Big)  = -\infty.\end{equation}
But, 
\begin{equation} \label{f} \PP \Big(  \sum_{i=k+1}^n \big(\lambda_i^d - f_M(\lambda_i)\big)>\delta p^d n^d \Big)\leq \PP\Big( \sum_{i=k+1}^n \lambda_i^d \Car_{\lambda_i\geq  M} >\delta n^dp^d\Big).\end{equation}
Let $R>0$, and denote by $\mathcal{N}[R,+\infty)$ the number of eigenvalues of $X_\sigma$ in $[R,+\infty)$. By Weyl's inequalities (see \cite[Theorem III.2.1]{Bhatia}), we have for any $i\geq 2$,
\begin{equation} \label{Weyl} \lambda_i(\hat X) \leq \lambda_i(X) \leq \lambda_{i-1}(\hat X),\end{equation}
where $\hat X = X -p u u^*$, with $u$ being the all-ones vector.
Therefore,
$$ \PP\big( \mathcal{N}[R,+\infty) \geq k \big) \leq \PP\big( |\{ i : (\lambda_i(\hat X))_\sigma \geq R \}| \geq k-1 \big).$$
As $\EE || \hat X|| =O( \sqrt{np})$ again by \cite[Example 4.10]{LHY}, we deduce by Proposition \ref{probatrou} that if $R$ is such that $\sqrt{np} = o(R)$, then for any $k\geq 2$, 
\begin{equation} \label{probtrouER} \PP( \mathcal{N}[R,+\infty) \geq k)\leq  2 \exp \Big( - \frac{R^2k}{32}\Big).\end{equation}
Now, take $R =  \sqrt{v_n/\beta_n}$, with $\beta_n =o( k)$ and $\beta_n = o(np \log(1/p))$.
Our choice of $R$ satisfies $\sqrt{np} = o(R)$, and we have
$$ \lim_{n\to +\infty} \frac{1}{v_n} \log \PP( \mathcal{N}[R/2,+\infty) \geq k)= - \infty.$$
Since we assumed that $\log^4(1/p) = o(k)$, we can actually set $\beta_n$ to satisfy the conditions,
\begin{equation} \label{defbetan} \log^4(1/p) = o(\beta_n), \ \beta_n =o(k) \text{ and } \beta_n = o(np \log(1/p)),\end{equation}
using the fact that $np$ grows polynomially fast to $+\infty$.
Thus, in view of \eqref{expoequitrunER} and \eqref{f}, it is sufficient to prove that,
\begin{equation} \label{condsuff} \limsup_{n\to +\infty} \frac{1}{v_n} \log \PP\Big( \sum_{i=k+1}^n \lambda_i^d \Car_{M\leq  \lambda_i \leq  R/2} >\delta n^dp^d\Big) = -\infty.\end{equation}
If $R/2< M$, then there is nothing left to prove. Assuming that $R/2 \geq  M$, we are going to use a ``slicing argument'' to fill the gap between these two levels. To this end, define an increasing sequence $M_m$ of levels such that $M_0 = M$, and for any $m\geq 0$,
$$ M_{m+1}^{d}=R^{d-2} M_m^2,$$
that is
$$\log M_m =   \Big( \frac{2}{d}\Big)^m \log M +\Big(1- \Big( \frac{2}{d}\Big)^{m}\Big) \log R.$$
Let $N$ be such that $ (2/d)^N\leq \log 2/\log R$. As $\log R = O(\log n)$, it is possible to find such $N$ with $N =O(\log \log n)$. For such $N$, we get $ M_N \geq R/2$.
Therefore,
$$\PP\Big( \sum_{i=k+1}^n \lambda_i^d \Car_{M\leq \lambda_i \leq R/2} >\delta n^dp^d\Big) \leq \PP\Big( \sum_{i=k+1}^n \lambda_i^d \Car_{M\leq \lambda_i \leq M_N} >\delta n^dp^d\Big).$$
Using a union bound we get,
\begin{align}
\PP\Big( \sum_{i=k+1}^n \lambda_i^d \Car_{M\leq  \lambda_i \leq  M_N} >\delta n^dp^d\Big) & \leq \PP\Big( \sum_{m=0}^{N-1}M_{m+1}^d\mathcal{N}[M_m,+\infty) >\delta n^dp^d\Big)\nonumber \\
& \leq \sum_{m=0}^{N-1} \PP\big( \mathcal{N}[M_m,+\infty)  >\delta n^dp^d/NM_{m+1}^d\big).\label{unionb}
\end{align}
As $M_m \geq M$ and $\sqrt{np} =o(M)$ by definition of $M$ in \eqref{defM}, we deduce from \eqref{probtrouER} that for any $m \geq 0$,
$$  \PP\Big( \mathcal{N}[M_m,+\infty) >\delta n^dp^d/NM_{m+1}^d\Big) \leq 2\exp\Big( -\frac{\delta n^d p^d}{32 N R^{d-2}}\Big).$$
But $R = \sqrt{v_n/\beta_n}$, therefore,
$$  \PP\Big( \mathcal{N}[M_m,+\infty)  >\delta n^dp^d/NM_{m+1}^d\Big) \leq2 \exp (- v_n c_n),$$
with 
$$c_n  = \frac{\delta \beta_n^{\frac{d-2}{2}}}{32N \log^\frac{d}{2} (1/p)}.$$
As $d/(d-2)\leq 3$ since $d\geq 3$, and $N=O(\log \log n)$, we see that with our choice of $\beta_n$ which satisfies $\log^4 (1/p) = o(\beta_n)$, we have in particular
$$ N^{\frac{2}{d-2}} \log ^{\frac{d}{d-2}}(1/p) = o(\beta_n).$$
Thus, $c_n$ goes to $+\infty$.
But, from the union bound \eqref{unionb}, we have
$$\PP\Big( \sum_{i=k+1}^n \lambda_i^d \Car_{M\leq  \lambda_i \leq M_N} >\delta n^dp^d\Big) \leq N \exp( - v_nc_n).$$
Therefore,
$$\lim_{n\to +\infty} \frac{1}{v_n} \log \PP\Big( \sum_{i=k+1}^n \lambda_i^d \Car_{M\leq  \lambda_i\leq  M_N} >\delta n^dp^d\Big)=-\infty,$$
which ends the proof of \eqref{concen}, and thus of the claim \eqref{firststep}.

Now, to end the proof of Lemma \ref{equiER}, we need the following lemma which estimates the expectation of the truncated trace involving the bulk eigenvalues.\renewcommand{\qedsymbol}{}
\end{proof}
\begin{Lem}\label{expect}
Assume $d\geq 4$ and $np^2 \gg 1$, or $d=3$, $\log n =o(np)$ and $k =o((np)^{3/2})$. Then,
$$ \big| \EE \big(  \tr (X^d) - g_k(X)\big) \big| = o(n^dp^d).$$

\end{Lem}
\begin{proof}We start with the case $d\geq 4$. 
By Weyl's inequalities \eqref{Weyl}, we get
$$ \big|  \tr (X^d) - g_k(X) \big|\leq n || \hat X||^d.$$
But from \eqref{tailsp}, we know that $\EE ||\hat X||^d = O( (np)^{\frac{d}{2}})$, thus
$$ \big|  \tr (X^d) - g_k(X) \big|= O( n (np)^{\frac{d}{2}}).$$
As $d\geq 4$ and $np^2 \to +\infty$, we have that $n^{\frac{d}{2}-1}p^{\frac{d}{2}}$ also goes to $+\infty$, which yields the claim.

Assume now that $d=3$, $\log n =o(np)$ and $k =o((np)^{3/2})$. As $\lambda_1(X)\geq 0$, we have,
$$ \tr (X^3) - \big( \tr_{[k]} (X_+^3) - \tr_{[k]}( X_-^3) \big)= \sum_{i=\ell+1}^{\ell'} \lambda_i(X)^3,$$
for some $\ell,\ell'$ such that $1 \leq \ell\leq k$, and $n-\ell'\leq k$. We have,
 $$\big|  \sum_{i=\ell+1}^{\ell'} \lambda_i(X)^3- \tr (\hat X)^3 \big| \leq\big|\sum_{i=\ell+1}^{\ell'} \lambda_i(X)^3-   \sum_{i=\ell+1}^{\ell'} \lambda_i(\hat X)^3 \big| + 2k ||\hat X||^3.$$
By Weyl's inequalities \eqref{Weyl}, we have for any $i\geq 2$,
$$\lambda_i(\hat X)^3 \leq  \lambda_i(X)^3\leq   \lambda_{i-1}(\hat X)^3.$$
Therefore,
$$0\leq \sum_{i=\ell+1}^{\ell'} \lambda_i(X)^3-   \sum_{i=\ell+1}^{\ell'} \lambda_i(\hat X)^3 \leq 2|| \hat X||^3.$$
Thus,
$$\big|  \EE \sum_{i=\ell+1}^{\ell'} \lambda_i(X)^3\big| \leq  | \EE \tr (\hat X)^3| + 2(k+1)\EE || \hat X||^3.$$
But, as $(\hat X)_{i,i} =-p$ for any $i\in\{1,\ldots,n\}$ and the off-diagonal entries of $\hat X$ are centered, we obtain
$$\EE \tr (\hat X)^3 = -np^3 + O(n^2p^2)=o(n^3p^3).$$
Besides, as $\log n =o(np)$, we know from \eqref{tailsp} that $\EE || \hat X ||^3 = O((np)^{\frac{3}{2}})$. Thus,
$$| \EE  \sum_{i=\ell+1}^{\ell'} \lambda_i(X)^3| = O( k(np)^{\frac{3}{2}}) + o(n^3p^3),$$
which gives the claim.

\end{proof}

\subsection{Proof of Proposition \ref{cycleup}}
From Lemma \ref{equiER}, we see that it suffices to consider the upper tail of the truncated trace,
$$g_k(X) = \begin{cases}
\tr_{[k]} X^d & \text{ if  $d$ is even,}\\
\tr_{[k]} X_+^d - \tr_{[k]} X_-^d & \text{ if  $d$ is odd.}
\end{cases}$$ 
for $k$ such that $ \log^4(1/p) = o(k)$, and with the additional condition that $k=o((np)^{3/2})$ in the case $d=3$. As we assume that $  \log^4(n) = o(np^2)$,
we can find $k$ which satisfies both conditions 
\begin{equation} \label{condk} k\log n =o(np^2 \log(1/p)), \text{ and }\log^4(1/p) = o(k),\end{equation} (the additional condition when $d=3$ being automatically fulfilled). 

With this choice of $k$, we will see that we have reduced the complexity enough to be able to apply Proposition \ref{NL}. Indeed, we will see that we can encode the ``gradient'' of the truncated trace $\tr_{[k]} X^d$ by $O(nk\log n )$ bits. Thus, the choice of $k$ made above will guarantee us that the main error term in Proposition \ref{NL} is negligible with respect to the large deviation speed.

First, we will check that the rate function that Proposition \ref{NL} is giving us, is as good as the one we are claiming, that is:

\begin{Lem}\label{psiphi}
Let $t\geq 1$ and define,
$$\psi_{n}(t) = \inf \big\{ \Lambda_p^*(Y) : g_k( Y) \geq t n^d p^d, Y \in \mathcal{H}_n^0 \big\}.$$
Then,   
$$ \psi_{n}(t) \geq  \phi_n\Big( t-O(k^{1-\frac{d}{2}})\Big),$$
where $\phi_n$ is defined in Proposition \ref{cycleup}.
\end{Lem}

\begin{proof}Note that when $d$ is even, the claim of the above lemma is trivial, so that we will assume for now on that $d$ is odd. Observe also that $\psi_{n}(t) =O(v_n)$ by taking $C$ to correspond to planting clique of size $r=\lceil (t-1)^{1/d} np \rceil$, that is, 
$$\forall i<j, \ C_{i,j}= \begin{cases}1 & \text{ if } i\wedge j \leq r,\\
p & \text{ if } i\vee j >r.
\end{cases}$$
Let then $Y\in \mathcal{H}_n^0$ be such that $\Lambda^*_p(Y) =O(v_n)$ and $g_k(Y) \geq t n^d p^d$. Without loss of generality we can assume that $Y_{i,j}\geq p$ for any $i<j$. By \cite[Corollary 3.5]{LZ}, we know that for any $x\geq 0$,
$$ \Lambda^*_p(p+x) \geq x^2 (\log (1/p) - o(1)).$$
We deduce, denoting by $U$ the matrix such that $U_{i,j} = 1$ for $i\neq j$ and null diagonal coefficients, that
$$ \tr (Y-pU)^2 =O(n^2p^2).$$
In particular, $\tr Y^2 = O(n^2p^2)$. We obtain for any $i\geq 1$,
$$ \lambda_i(Y_-) =O\Big( \frac{np}{\sqrt{i}}\Big).$$
Therefore,
$$ \sum_{i=k+1}^n \lambda_i(Y_-)^d  = O(n^dp^d k^{1-\frac{d}{2}}).$$
But,
$$ g_k(Y) \leq  \tr (Y^d)  +\sum_{i=k+1}^n \lambda_i(Y_-)^d,$$
which gives the claim.
\end{proof}

We can now proceed with the proof of the upper bound of Proposition \ref{cycleup}. Note that $\Phi$ is lower semi-continuous as one can see from its explicit expression \eqref{defPhi}. Therefore, by Lemma \ref{equiER} it is sufficient to show that for any $t> 1$,
$$ \limsup_{n\to +\infty} \frac{1}{v_n} \log \PP\big ( g_k(X)  \geq t n^d p^d\big) \leq - \Phi(t),$$
with $k$ such that $ \log^4(1/p)= o(k)$ and $ k =o(np^2)$.

 Fix some $t>1$ and $\delta>0$ such that $t-\delta> 1$.
Observe that as $\Lambda_p^*$ is strictly convex, $\psi_{n}$, defined in Lemma \ref{psiphi}, is strictly increasing.  
Therefore, we can apply Corollary \ref{CorNL}. Let $\mathcal{H}_n^0([0,1])$ denote the subset of $\mathcal{H}_n^0$ which consists of matrix with entries in $[0,1]$. Observe that for any $Y\in\mathcal{H}_n^0([0,1])$, 
$$\Lambda^*_p(Y) \leq n^2 \log(1/p),$$
 so that the tightness assumption \eqref{tightness} of Proposition \ref{NL} is fulfilled with $K = \mathcal{H}_n^0([0,1])$ and $\kappa=n^2\log (1/p)$. Arguing as in the proof of \eqref{difftr}, and using the fact that $||X||\leq n$ for any $X\in \mathcal{H}_n^0([0,1])$, we obtain that for any $X,Y \in \mathcal{H}_n^0([0,1])$,
$$g_k(X) - g_k(Y) \leq \sup_{ H \in W}  \tr H(X-Y),$$
where 
$$ W = \big\{ H \in \mathcal{H}_n : \rk(H) \leq 2k, \ || H|| \leq 2dn^{d-1}\big\},$$
using the fact that for any $H \in\mathcal{H}_n([0,1])$, $|| H|| \leq n$.
As the Lipschitz constant of $g_k$ on $\mathcal{H}_n([0,1])$  and the diameter of $\mathcal{H}_n([0,1])$ are both only polynomial in $n$, we get by Corollary \ref{CorNL}, 
$$\PP( g_k(X)\geq tn^dp^d ) \leq - \psi_{n}(t-\delta) + \log N(W, (\delta/2n) B_{2}) + O(\log n),$$
 where $N(W, (\delta/2n) B_{2})$ is the covering number of $W$ by balls of radius $\delta/n$ for the Hilbert-Schmidt norm. By Lemma \ref{complexity}, we have
$$\log  N(W, (\delta/2n) B_{2}) = O(nk \log n).$$
As we chose $k$ such that $k\log n  = o(np^2\log (1/p))$, we get
$$\PP( g_k(X)\geq tn^dp^d ) \leq - \psi_{n}(t-\delta) +o(v_n).$$
Using Lemma \ref{psiphi} and the fact that $\phi_n$ is increasing, we obtain that for $n$ large enough, $\psi_{n}(t-\delta) \geq \phi_n(t-2\delta)$. Therefore,
$$\limsup_{n\to +\infty} \frac{1}{v_n} \log \PP( g_k(X)\geq tn^dp^d ) \leq - \Phi(t-2\delta).$$
Since this is true for any $\delta>0$, and $\Phi$ is lower semi-continuous, we get the claim by taking $\delta \to 0$.

\section*{Acknowledgement} I would like to thank Ofer Zeitouni for his enlightening insight and advice. I am also grateful to Ronen Eldan for some influential discussions and for the remark \ref{remRonen}, as well as to Anirban Basak for many helpful discussions. Finally, I thank the anonymous referee for the numerous comments and suggestions.

%
%
%
%
%

\bibliographystyle{plain}
\bibliography{main}{}

\end{document}